\documentclass[14pt]{aims}
\usepackage{amsmath, amssymb,mathtools}
  \usepackage{paralist}
  \usepackage{graphics} 
  \usepackage{epsfig} 
\usepackage{graphicx}  \usepackage{epstopdf}
 \usepackage[colorlinks=true]{hyperref}
 \hypersetup{urlcolor=blue, citecolor=red}
 \usepackage[toc,page]{appendix}
\usepackage{chngcntr}

  \def\currentyear{}
   \def\currentmonth{}
    \def\ppages{}
     \def\DOI{}
\numberwithin{equation}{section}

\newtheorem{theorem}{Theorem}[section]
\newtheorem{corollary}{Corollary}

\newtheorem{lemma}[theorem]{Lemma}
\newtheorem{proposition}[theorem]{Proposition}
\theoremstyle{definition}

\newtheorem{remark}{Remark}

\newcommand{\RR}{\mathbb{R}}

\newcommand{\pd}{\partial}
\font \dsrom=dsrom10 scaled 1200
\def \indic{\textrm{\dsrom{1}}}
\title[Stability of infinite time aggregation] 
      {Stability of infinite time aggregation for the Patlak-Keller-Segel equation}

\author[T. Ghoul and N. Masmoudi]{}

 \email{masmoudi@cims.nyu.edu}
 \email{teg6@nyu.edu}

\begin{document}
\maketitle

\centerline{\scshape Tej-Eddine Ghoul}
\medskip
{\footnotesize
 \centerline{Department of Mathematics, New York University in Abu Dhabi,}
   \centerline{Saadiyat Island, P.O. Box 129188, Abu Dhabi, United Arab Emirates.}
}

\medskip

\centerline{\scshape Nader Masmoudi}
\medskip
{\footnotesize
  \centerline{Department of Mathematics, New York University in Abu Dhabi,}
   \centerline{Saadiyat Island, P.O. Box 129188, Abu Dhabi, United Arab Emirates.}
	\centerline{Courant Institute of Mathematical Science, New York University,}
   \centerline{251 Mercer Street, New York, NY 10012, USA.}
}

\bigskip

\begin{abstract} We consider the parabolic-elliptic Patlak-Keller-Segel (PKS) model of chemotactic aggregation in two space dimensions which describes the aggregation of bacteria under chemo-taxis. When the mass is equal to $8\pi$ and the second moment is finite (the doubly critical case), we give a precise description of the dynamic as time goes to infinity and extract the limiting profile and speed. The proof shows that this dynamic is stable under perturbations.
\end{abstract}

\section{Introdution}
We consider the two dimensional Patlak-Keller-Segel system,
\begin{eqnarray}\label{PKellerseguel}
\left\{
\begin{array}{ll}
\partial_t u(x,t) -\Delta u(x,t)&=-\nabla\cdot\bigl(u(x,t)\nabla c(x,t)\bigr) \\
-\Delta c(x,t) &=  u(x,t)\\
u(x,t=0)&=u_0\geq 0,
\end{array}
\right.
\end{eqnarray}
where $x\in\RR^2$ and $t>0$.
This system is generally considered the fundamental mathematical model for the study of aggregation by chemotaxis of certain microorganisms \cite{Patlak,KS,Hortsmann,HandP}. From now on we will refer to \eqref{PKellerseguel} as Patlak-Keller-Segel (PKS).
The first equation describes the motion of the microorganism ($u$ represents the density of cells) as a random walk with drift up the gradient of the \emph{chemo-attractant} $c$. The second equation describes the production and (instantaneous) diffusion of the chemo-attractant. 
PKS and related variants have received considerable mathematical attention over the years, for example, see the review \cite{Hortsmann} or some of the following representative works \cite{ChildressPercus81,JagerLuckhaus92,Nagai95,Biler95,HerreroVelazquez96,Senba02,Biler06,BlanchetEJDE06,Senba07,Blanchet08,BlanchetCalvezCarrillo08,BlanchetCarlenCarrillo10}. The equation is $L^{\frac{d}{2}}$ critical in dimension d, hence it is mass or $L^1$ critical in dimension 2, namely: if $u(t,x)$ is a solution to \eqref{PKellerseguel} then for all $a \in (0,\infty)$, so is
\begin{equation*}
u_a(t,x) = \frac{1}{a^2}u\left(\frac{t}{a^2}, \frac{x}{a}\right). 
\end{equation*}
Moreover, if $M=8\pi$ then the second moment is also conserved. This justifies the terminology doubly critical if $M=8\pi$ and yields very inetresting dynamical properties.
It has been known for some time that  \eqref{PKellerseguel} possesses a critical mass: if $\|{u_0}\|_1 \leq 8\pi$ then classical solutions exist for all time (see e.g. \cite{Senba02,BlanchetEJDE06,Blanchet08,BlanchetCarlenCarrillo10,BlanchetCalvezCarrillo08}) and if $\|u_0\|_1 > 8\pi$ then all classical solutions with finite second moment blow up in finite time \cite{JagerLuckhaus92,Nagai95,BlanchetEJDE06} and are known to concentrate at least $8\pi$ mass into a single point at blow-up \cite{Senba02} (see also \cite{HerreroVelazquez96,Senba07}). 
Another important property of \eqref{PKellerseguel} that plays a decisive role in our work is the existence (and uniqueness) of self-similar spreading solutions for all mass $M \in (0,8\pi)$. These are known to be global attractors for the dynamics if the total mass is less than $8\pi$ \cite{BlanchetEJDE06} and for the purposes of our analysis.
The mass is conserved in \eqref{PKellerseguel}:
\begin{equation}\label{consmass}
\frac{d}{dt}\int_{\RR^2}u(x,t)dx=0,
\end{equation}
then we set 
$$M_0=\int_{\RR^2}u_0(x)dx.$$
Notice that the center of mass is also conserved
$$\frac{d}{dt}\int_{\RR^2}x_iu(x,t)dx=0 \mbox{ for } i\in\{1,2\}.$$
Since the solution of the Poisson equation $-\Delta c= u$ is given by
$$c=\phi_u =-\frac{1}{2\pi}\log|\cdot|\star u,$$
hence, the system becomes:
\begin{eqnarray}\label{a}
\left\{
\begin{array}{ll}
u_t &=\nabla \cdot (\nabla u - u\nabla\phi_u) \\
\phi_u &=-\frac{1}{2\pi}\log|\cdot|\star u  \mbox{     in } [0,T]\times\RR^2\\
u(0) &= u_0\geq0.
\end{array}
\right.
\end{eqnarray}
In addition, we have that the flow dissipates the free energy
\begin{align}\label{freenergy}
\frac{d}{dt}\mathcal{F}(u)=\frac{d}{dt}\Big(\int_{\RR^2}u\log udx-\frac{1}{2}\int_{\RR^2}u\phi_u dx\Big)\leq 0,
\end{align}
where $\mathcal{F}$ is the sum of the entropy and the potential energy.
The problem is locally well-posed in the finite borel measures space \cite{BM}, but the question of global existence requires more conditions.
Indeed, we can formally compute a virial indentity for a solution of \eqref{a} 
\begin{align*}
\frac{d}{dt}\int_{\RR^2}u(x,t)|x|^2dx&=-\int_{\RR^2}2x\cdot\nabla udx+\frac{1}{2\pi}\int_{\RR^2\times\RR^2}\frac{2x\cdot(y-x)}{|x-y|^2}u(x,t)u(y,t)dydx\\
&=4M_0+\frac{1}{2\pi}\int_{\RR^2\times\RR^2}\frac{(x-y)\cdot(y-x)}{|x-y|^2}u(x,t)u(y,t)dydx\\
&=4M_0-\frac{M_0^2}{2\pi}.
\end{align*}

So that all solutions with finite second moment and mass bigger than $8\pi$ cannot be global in time, and we also have an estimation on the maximal existence time $T_{Max}$ if the second moment of the initial data $u_0$ is finite,
$$T_{Max}\leq \frac{2\pi}{M_0(M_0-8\pi)}\int_{\RR^2}|x|^2u_0(x)dx.$$
Thus, this virial computation is a heuristic explanation of the following trichotomy :
\begin{itemize}
\item If $\int_{\RR^2}u_0(x)dx<8\pi$ solutions are global in time, and the solutions are spreading \cite{BDP},
and  the density converges to a self-similar profile in rescaled variables.
\item If $\int_{\RR^2}u_0(x)dx>8\pi$ solutions blow up in finite time, and we have  aggregation in finite time\cite{BDP,JL}.
\item And if $\int_{\RR^2}u_0(x)dx=8\pi$ and $\int_{\RR^2}|x|^2u_0(x)dx<\infty$ we have global existence and the solution $u$ concentrates  at the origin in infinite time \cite{BCM}, whereas if $\int_{\RR^2}u_0(x)dx=8\pi$ and $\int_{\RR^2}|x|^2u_0(x)dx=\infty$ what could happen is not clear except if the solution initially is sufficiently close to a rescaling $Q_a$ of the stationnary solution $Q$  \cite{BCC}.
\end{itemize}
Indeed, the stationnary solution of \eqref{a} is 
$$Q(x)=\frac{8}{(1+|x|^2)^2},$$
its mass is $8\pi$, its second moment is infinite, and $Q$ is the unique minimizer up to symmetries of the free energy $\mathcal{F}(u)$
with $\int_{\RR^2}u(x,t)dx=8\pi$.

When $M_0<8\pi$ the solutions are spreading and they converge to a self-similar profile in rescaled variables with a rate $\frac{1}{t}$ in \cite{CD,BDEF,BKLN}. Whereas if $M_0>8\pi$ the solutions blow up in finite time and there are few results about it. One of them describes the dynamic of the blowup \cite{HV}.
However, there is another important one in \cite{RS} where they proved the stablity of the blowup in addition to the description of the dynamic, but all of these results are for radial solutions.
Indeed, in \cite{RS} they proved 
\begin{itemize}
\item Universality of the blow-up profile:
for all $t\in[0,T)$ $u(x,t)=\frac{1}{\lambda(t)^2}(Q+\varepsilon)(t,\frac{x}{\lambda(t)})$ with  $\|\varepsilon\|_{H^2_Q}\longrightarrow 0 $ as  $t\rightarrow T,$
where $\lambda(t)=\sqrt{T-t}e^{-\frac{\sqrt{|\log(T-t)|}}{2}+O(1)}$
\item Stability of the dynamic under small perturbation in $H^2_Q(\RR^2)\cap L^1(\RR^2)$.
\end{itemize} 

Their proof is based on the modulation theory which is a strong method for critical problems. For example, it allows people to prove the stability and to describe the dynamic of the flow for the Nonlinear Schrodinger equation when the mass is critical \cite{MR1,MR2}, for the Schrodinger map when the energy is critical \cite{MRR}. And also for the 1-corotational harmonic heat flow map when the energy is critical \cite{RS2}.

In this paper we are interested in the critical mass case namely $M_0=8\pi$. 
The problem when $M_0=8\pi$ is also energy critical
$$\mathcal{F}(u_\lambda)=\mathcal{F}(u),$$
and note that we have in this case one more conservation law 
\begin{equation}\label{eqn:cons2moment}
\frac{d}{dt}\int_{\RR^2}u(x,t)|x|^2dx=0.
\end{equation}

On one hand if $M_0=8\pi$, $\int_{\RR^2}u_0(x)|x|^2dx=+\infty$ and the initial data $u_0$ is sufficiently close to a rescaling $Q_a$ of $Q$ then the solution converges to $Q_a$ when $t\rightarrow +\infty$\cite{BCC}. On the other hand, if $M_0=8\pi$ and $\int_{\RR^2}u_0(x)|x|^2dx<+\infty$ the solution exists globally and concentrates in infinite time \cite{BCM}. The proof is based on a contradiction argument and a virial identity.
When the domain is bounded in \cite{KS}, the authors proved that radial solutions $u$ collapse in $0$  with the following rate
$$u(0,t)=8e^{\frac{5}{2}+2\sqrt{2t}}\Big(1+O(t^{-\frac{1}{2}}\log (4t))\Big),$$
as $t\longrightarrow +\infty$.
The authors in \cite{KS} used the partial mass equation which remove the nonlocal difficulity and an argument of subsolutions and supersolutions to bound the solution from above and below. 
When the domain is unbounded the infinite speed of propagation of the heat semi group will send some mass to infinity almost instantly, and the concentration will occur only when all the mass $8\pi$ is present again in the center of mass. While all the mass concentrate at the origin when $t\rightarrow +\infty$, the second moment is ejected to infinity in space. The goal of the paper is to describe the dynamic in $\RR^2$ when the mass is critical $M_0=8\pi$ and to prove the stability of that dynamic under pertubations of the initial data.

\begin{theorem}\label{theorem:main}
For all $A_0>0$, there exists $\mu^*_0>0$ such that for all $\mu_0<\mu^*_0$ and $u_0$ initial data of the form 
\begin{align}
u_0=\frac{1}{\mu_0}\Big(Q\Big(\frac{x}{\sqrt{\mu_0}}\Big)e^{-\frac{x^2}{2}}+\tilde{\varepsilon}_0\Big(\frac{x}{\sqrt{\mu_0}}\Big) \Big) \mbox{ where },  \int_{\RR^2}\frac{|\tilde{\varepsilon}_0(y)|^2}{Q}e^{\frac{|y|^2}{2}}dy<A_0\mu_0,
\end{align}
and
\begin{align}\label{cond_on_u0}  
\int_{\RR^2}u_0(x)dx=8\pi,\quad \int_{\RR^2}u_0(x)x_idx=0 \quad \mbox{ for all } i\in\{1,2\},\quad I:=\int_{\RR^2}u_0(x)|x|^2dx,
\end{align}
the corresponding solution $u$ blows up in infinite time and satisfies
\begin{itemize}
\item Universality and stability of the profile: There exists $C>0$ and $\lambda\in C^1(\RR_+)$ such that
for $t$ large enough, we have
 \begin{align}\label{decomp_u}
 u(x,t)=\frac{1}{\lambda(t)^2}\Big(Q\Big(\frac{x}{\lambda(t)}\Big)e^{-\frac{|x|^2}{2t}}+\tilde{\varepsilon}\Big(t,\frac{x}{\lambda(t)}\Big)\Big),
 \end{align}
with  $$\int_{\RR^2}\frac{|\tilde{\varepsilon}(t,y)|^2}{Q(y)}e^{\frac{|y|^2}{2t\log(2t+1)}}dy\leq \frac{C}{t\log(t)} $$ 
where $\lambda(t)=\frac{\sqrt{I}}{\sqrt{\log(2t+1)+O(\log(\log(t))}}$.
\end{itemize}
\end{theorem}

\begin{remark}
The proof requires a spectral gap bound ( see Corollary \ref{corollary:NT} ) which follows from the study of the spectrum of a linearized operator \cite{CD}. As of now the proof of Corollary \ref{corollary:NT} in the nonradial case requires some eigenvalue bounds which were only done numerically in \cite{CD}. We hope to address this issue rigorously elsewhere.
\end{remark}

\begin{remark}
Notice that if we select any rescaling of $u_0$ initially like $(u_0)_a(x)=\frac{1}{a^2}u_0(\frac{x}{a})$ for some $a>0$, then the solution $u_a(t,x)=\frac{1}{\lambda(\frac{t}{a^2})^2a^2}(Q(\frac{x}{a\lambda(\frac{t}{a^2})})e^{-\frac{|x|^2}{2t}}+\tilde{\varepsilon}(\frac{t}{a^2},\frac{x}{a\lambda(\frac{t}{a^2})}))$  induced by $(u_0)_a$ still behaves asymptotically as in Theorem \ref{theorem:main}.
\end{remark}
In the radial case, we can remove the smallness assumption on the initial data, using principle maximum on the partial mass equation, namely:
\begin{corollary}\label{corollary:th_on_mass}
Let $u_0$ be a radial initial data satisfying \eqref{cond_on_u0} and
\begin{align}\label{boundu_0}
\int_{\RR^2}\frac{|u_0|^2}{Q}e^{\frac{|y|^2}{2}}dy<+\infty.
\end{align}
Then there exist $C_1>0$, $C_2>0$, $\lambda_1(t)$ and $\lambda_2(t)$ such that the partial mass $m_u$ of the corresponding solution $u$ (see \eqref{pmass}) satisfies for $t$ large enough and $r=|x|$
\begin{align}\label{part_mass_rate}
-\frac{C_1\lambda_1^2e^{-\frac{C_1 r^2}{2t}}}{\lambda_1^2+r^2}-\frac{C_1\lambda_1^2}{(\lambda_1^2+r^2)t|\log t|}\leq 8\pi- m_u(r,t)\leq\frac{C_2\lambda_2^2e^{-\frac{C_2 r^2}{2t}}}{\lambda_2^2+r^2}+\frac{C_2\lambda_2^2}{(\lambda_2^2+r^2)t|\log t|},
\end{align} 
with $\lambda_i(t)=\frac{I_i}{\sqrt{\log(2t+1)+O(\log(\log(t))}}$, and $I_i>0$ for $i\in\{1,2\}$.
\end{corollary}

The proof of the corollary is given in Section 8.
\subsection{Strategy of the proof}

Let's explain the main ideas of the proof that we think could be applicable to many other critical problems.
The main idea of the proof is to pass by the subcritical case to reach the critical one.
Indeed, the blowup profile of the solution $u$ of \eqref{a} is the unique stationary solution $Q=\frac{8}{(1+|x|^2)^2}$, but we cannot use the linearisation around that stationary solution because we need the second moment of $u$ to be finite and the mass to be $8\pi$. Since, $Q$ has an infinite second moment and $8\pi$ mass, it seems natural to multiply $Q$ by a cutoff function $\chi$.
However, $\chi Q$ doesn't satisfy a good equation and has an uncontrollable error.
The idea is to notice that the stationary solution of the subcritical case after rescaling converge uniformly to $Q$.
Indeed, if we rescale \eqref{a} with the self similar variables when the mass $M$ is subcritical there exists for that $M$ a unique stationary solution $n_\infty$ which depends on $M$ and decays exponentially at infinity. Actually, after another rescaling we prove that $n_\infty$ converges uniformly to $Q$ when $M\rightarrow8\pi$.

Actually, we prove that $n_\infty$ after a certain rescaling behaves as $Q$ multiplied by a cutoff function, but the advantage is that $n_\infty$ is a stationary solution of some related equation \eqref{ninfiny}.
The main difficulty now is that everything depends on the mass, i.e. the linearized operator and its spectrum, the energy and the norms that we are controlling. And since the mass is not fixed, we need that all the inequalities are uniform with respect to the mass.
Then, it appears a big difficulty in this problem, the energy that we are controlling is coercive but not uniformly with respect to the mass. To overcome this problem, we discovered an other bound on the energy coming from a property of the nonlinear term (see Proposition \ref{proposition:boundonphieps}).

Classically, to describe the dynamic of the blowup we need to rescale the solution, and the best rescaling is the blowup speed $\lambda$.
Then we renormalize the previous equation with the following change of variable 
$u(t,x) = \frac{1}{\lambda^2(t)} v(s,y)$, $\frac{ds}{dt}= \frac{1}{\lambda^2(t)}$, and $y= \frac{x}{\lambda(t)}$,
where $\lambda$ is an unknown of our problem and will be determined later.
Then we obtain 
\begin{eqnarray}\label{b}
v_s- \frac{\lambda_s}{\lambda}\Lambda v &=\nabla \cdot (\nabla v - v\nabla\phi_v).
\end{eqnarray}
Where $$\Lambda_y v =\nabla\cdot(yv).$$
Actually, we don't rescale  \eqref{a} by $\lambda$ directly, but we arrive to \eqref{b} in 2 steps. In a first step we rescale by the self-similar variables  to make the stationnary solution appear of the subcritical case. And in a second step we rescale by a certain parameter $\mu$ that will be defined later such that the combination of these two rescalings leads to \eqref{b}.  

We first start in the second section to set the equations and the bootstrap argument. In the third section we prove the spectral gap estimate. In the fourth section we construct an approximate solution and in the fifth section we do the derivation of the blowup speed $\lambda$.  In the sixth section we derive a uniform bound on the potential and in the last section we proof the rigorous control of the error $\varepsilon$.
\begin{remark}
We think that the new method we present here that is solving the critical problem by passing to the limit in the subcritical problem is applicable to many others critical problems.
Our result is true for non radial solution under a spectral gap assumption, but in the radial case the spectral gap has been proved in \cite{CD}.
Moreover, the speed of concentration that we derive rigorously from the conservation of the second moment is the one found formally by the physicist \cite[Section 3.3.3]{CS}.
\end{remark}
\subsection{Notations and Conventions}
We will use the following conventions:
\begin{itemize}
\item $C$ denotes a constant that may change from one line to the next and is uniform with respect to $\mu$.
\item For functions $f$ and $g$, we denote $f=O(g)$ if there exists $C$ uniform with respect to $\mu$ such that $f\leq Cg$.
\item For quantities $A$ and $B$, we denote $A \lesssim B$ if there exists $C$ uniform with respect to $\mu$ such that $A \leq CB$.
\item $<r>$ denotes $\sqrt{1+r^2}$.
\item For radial functions such as $Q,\phi_Q$ we keep the same notation for $Q(x)$ and $Q(r)$ where $r=|x|$ and write $\Delta\phi=\pd_{rr}\phi+\frac{1}{r}\pd_r\phi$.
\end{itemize}

\section{Bootstrap argument}
\subsection{The subcritical problem}

We expect $u$ to behave like $Q$ at the blowup time.
Since the second moment of $Q$ is infinite and we are looking for solutions with finite second moment linearizing $u$ arround $Q$ (i.e $u=Q+\varepsilon$) is not helpful.
And if we multiply the ground state $Q$ by a cutoff function $\chi$ to have a finite second moment then the error of our solution $Q\chi$ becomes too large and we cannot close our estimates.
However, if we consider the stationnary solution of the subcritical mass problem $n_\infty$ and we rescale it by a certain parameter $\mu$ then the $n_\infty$ rescaled by $\mu$ will converge to $Q$.
And this sequence $Q_\mu$ has much better properties as we will see later.
To make $n_\infty$ appear,  we rescale \eqref{a} by the self-similar variables,  using $z=\frac{x}{R(t)}$ and $\tau=\log R(t)$ with $R(t)=\sqrt{1+2t}$ such that $u(x,t)=\frac{1}{R(t)^2}w(\frac{x}{R},\log R):$ 
\begin{eqnarray}\label{eqinz}
\left\{
\begin{array}{ll}
\partial_\tau w -\Delta w-\Lambda w &=-\nabla\cdot\bigl( w\nabla \phi_w \bigr) \\
-\Delta \phi_w &=  w,
\end{array}
\right.
\end{eqnarray}
where $\Lambda w=\nabla\cdot(zw)$ and $$\int_{\RR^2}w(z,\tau)dz=\int_{\RR^2}u(x,t)dx=8\pi =\int_{\RR^2}u_0(x)dx.$$
It is well known that \eqref{eqinz} has stationnary solution $n_\infty$ of mass $M$ for all $M<8\pi$.
Actually $n_\infty$ is the solution of the following equation 
\begin{align}\label{ninfiny}
\left\{
\begin{array}{ll}
\Delta n_\infty +\nabla\cdot( z n_\infty - n_\infty\nabla\phi_{n_\infty})=0 \\
\phi_{n_\infty} =-\frac{1}{2\pi}\log|\cdot|\star n_\infty  \mbox{     in } \RR^2.
\end{array}
\right.
\end{align}
The equation can also be written in this way 
\begin{align}\label{eqninfiny}
-\Delta \phi_{n_\infty}= n_\infty = M_{n_\infty} \frac{e^{\phi_{n_\infty}-\frac{|z|^2}{2}}}{\int_{\RR^2}e^{\phi_{n_\infty}-\frac{|z|^2}{2}}dz}.
\end{align}
The linearized operator around $n_\infty$ is also well known:
\begin{align}\label{Lenz}
\mathcal{L}^zw&=\Delta w +\nabla\cdot( z w - w\nabla\phi_{n_\infty}-n_\infty\nabla\phi_w)=\nabla\cdot\Big(n_\infty\nabla\Big(\frac{w}{n_\infty}-\phi_w\Big)\Big)\nonumber\\
&=\nabla\cdot\Big(n_\infty\nabla(\mathcal{M}^z w)\Big), 
\end{align}
$$\mathcal{M}^z w=\frac{w}{n_\infty}-\phi_{w}.$$
Notice that $n_\infty$ can be parametrized by its mass or by its maximum at $z=0$. To emphasize the depence of $n_\infty$ in $M$ we will denote $n_\infty$ as $n_\infty^M$. And $M$ during all the paper will denote the mass of $n_\infty$. We will  parametrize $n_\infty$ by its maximum, we set $n_\infty(0)=\frac{8}{\mu}$, then if $\mu$ goes to 0 the mass $M=M(\mu)$ will go to $8\pi$.
Hence, if we rescale $n_\infty$ correctly with $\mu$, the maximum of the rescaled $n_\infty$ becomes $8$ which is the maximum of $Q$, and this rescaled $n_\infty$ converges uniformly to $Q$ as $\mu$ goes to zero.
The following proposition summarizes the previous statements:
\begin{proposition}
Let $\mu=\frac{8}{n_\infty^M(0)}$ and $y=\frac{z}{\sqrt{\mu}}$, then $\mu n_\infty^M(\sqrt{\mu} y)\rightarrow Q(y)$ uniformly as $\mu\rightarrow 0$. Moreover,
$$Q(y)e^{-\frac{\mu|y|^2}{2}}\leq \mu n_\infty^M(\sqrt{\mu} y)\leq Q(y), \mbox{ for all } y\in\RR^2 .$$
\end{proposition}
This is a direct consequence of Proposition \ref{proposition:comp1}.
Since $\mu n_\infty^M(\sqrt{\mu}\cdot)$ converges to $Q$, it seems natural to rescale \eqref{eqinz} by using
$y=\frac{z}{\sqrt{\mu}}$,$\frac{ds}{d\tau}=\frac{1}{\mu}$ with $w(z,\tau)=\frac{1}{\mu}v(\frac{z}{\sqrt{\mu}},s)$ where $\mu$ will be fixed later:
\begin{eqnarray}\label{v2rescal}
\left\{
\begin{array}{ll}
\partial_s v -\Delta v-(\mu+\frac{\partial_s \mu}{2\mu})\Lambda v &=-\nabla\cdot\bigl( v\nabla \phi_v \bigr) \\
-\Delta \phi_v =  v.
\end{array}
\right.
\end{eqnarray}
To resume the situation we have 3 set of variables $(x,t); (z,\tau); (y,s)$ which are linked in the following way:
\begin{align}
z=\frac{x}{R(t)}; \hspace{0.5cm} y=\frac{z}{\sqrt{\mu}}=\frac{x}{\sqrt{\mu}R(t)};  \\
\tau(t)=\log(R(t)), \hspace{0.5cm} \frac{ds}{d\tau}=\frac{1}{\mu}; \hspace{0.5cm} \frac{ds}{dt}=\frac{1}{\mu R(t)^2};  
\end{align}
where $R(t)=\sqrt{1+2t}$. We fix $s(0)=e$ to avoid problem if one divides by $s$ or if one takes the $\log$ of $s$.
Now we set $$Q_\mu(y)=\mu n_\infty^M(\sqrt{\mu} y),$$
hence $Q_\mu$ solves 
\begin{align}\label{eq_Q_mu}
\Delta Q_\mu-\nabla\cdot(Q_\mu\nabla\phi_{Q_\mu})=-\mu\nabla\cdot(yQ_\mu).
\end{align}
Let $v=Q_\mu+\varepsilon$, we will prove later that  $M(\mu)=\int_{\RR^2}Q_\mu dy= 8\pi+2\mu\log(\mu)+O(\mu)$ which implies that the mass of $\varepsilon$ will be of order $\mu\log(\mu)$ since the mass of $v$ is $8\pi$. We need $\varepsilon $ to be at least of order $\mu$ to close our estimates hence, we add a correction to $Q_\mu$. 
We set $\tilde{Q}_\mu=Q_\mu+T_\mu (y)$ where $T_\mu$ is a correction  to insure that $\tilde{Q_\mu}$ is of mass $8\pi$. It will be constructed in section 4.
 
\subsection{Setting up the bootstrap argument} 

We consider an initial data of the form
$$ v_0 = \tilde{Q}_{\mu_0}+\varepsilon_0, $$
with $ 8\pi $ mass, a finite second moment $I$ such that, $\int_{\RR^2}\varepsilon_0 dy=0$ and where $\mu_0>0$ is sufficiently small.
We will show  the existence and uniqueness of the following decomposition of our solution 
\begin{align}\label{decomp_v}
v(y,s)=\tilde{Q}_\mu+\varepsilon,
\end{align}
where the mass of $\varepsilon$ is zero.
Set
\begin{align}\label{L2qmu}
L^2_{Q_\mu}(\RR^2)=\Big\{f\in L^2(\RR^2) \mbox{ such that } \int_{\RR^2}\frac{f^2}{Q_\mu}dx<\infty\Big\},
\end{align}
and
\begin{align}\label{L2qmu0}
L^2_{{Q_\mu},0}(\RR^2)=\Big\{f\in L^2_{Q_\mu}(\RR^2) \mbox{ such that } \int_{\RR^2} f(x)dx=0\Big\}.
\end{align}
In the following Lemma we fix $\mu$ such that $(\varepsilon,|\cdot|^2)_{L^2}=0$.
\begin{lemma}\label{lemma:mod}[Modulation]
For all $A>0$, there exists $\bar{\mu}>0$ such that for all $\mu_*\in(0,\bar{\mu})$, and $v\in L^2_{Q_{\mu_*}}$.
If
 \begin{align}\label{modcond}
 \|v-\tilde{Q}_{\mu_*}\|_{ L^2_{Q_{\mu_*}}}\leq \delta(\mu_*):=\sqrt{A\mu_*},
 \end{align}
 there exist a unique $\mu>0$ and a unique $\varepsilon$ such that 
$$v(y)=\tilde{Q}_\mu(y)+\varepsilon(y),$$
and
$$(\varepsilon,|\cdot|^2)_{L^2}=0.$$
In addition,
$$ \|\varepsilon\|_{L^2_{Q_{\mu_*}}}\lesssim \delta(\mu_*),$$
and
$$|\mu-\mu_*|\lesssim \mu_*^{\frac{3}{2}}|\log(\mu_*)|^{\frac{1}{2}}.$$
\end{lemma}
\begin{proof}[Proof of Lemma \ref{lemma:mod}]
The proof is based on a careful use of the Intermediate Value Theorem.
Given $v$ satisfying \eqref{modcond}, consider the $C^1$ function 
$$ F(\mu)=(v-\tilde{Q}_{\mu},|\cdot|^2)_{L^2}.$$
From \eqref{2momdmutildeqmu} we deduce
$$\partial_{\mu} F_{|\mu=\mu_*}=((\partial_{\mu} \tilde{Q}_{\mu})_{|\mu=\mu_*},|\cdot|^2)=-\frac{C}{\mu_*}+O(|\log\mu_*|^2)\neq0,$$
and from \eqref{2momd2mutildeqmu}
$$|\partial_{\mu\mu}F|\lesssim \frac{1}{\mu^2}.$$
It follows for $|\mu-\mu^*|\leq\frac{\mu^*}{2}$,
$$F(\mu)=F(\mu_*)+\partial_\mu F(\mu_*)(\mu-\mu_*)+O(\frac{1}{\mu^2_*})(\mu-\mu_*)^2,$$
and since 
$$F(\mu_*)\leq \|v-\tilde{Q}_{\mu_*}\|_{ L^2_{Q_{\mu_*}}}\Big(\int_{\RR^2}|y|^2\tilde{Q}_{\mu_*}dy\ Big)^{\frac{1}{2}}\lesssim \sqrt{\mu_*}|\log\mu_*|^{\frac{1}{2}},$$
we deduce that there exists $C^*$ such that:
$$\mu_{\pm}=\mu_*\pm C^*\mu_*^{\frac{3}{2}}|\log\mu_*|^{\frac{1}{2}},$$
verify
$$F(\mu_-)<0,\quad \mbox{ and } \quad F(\mu_+)>0.$$
Hence, by the Intermediate value theorem there exists $\mu\in(\mu_-,\mu_+)$ such that $$F(\mu)=0,$$ 
which concludes the proof. The proof of \eqref{2momdmutildeqmu} and \eqref{2momd2mutildeqmu} is given in Section 5.
\end{proof}
Let $u_0$ be as in Theorem \ref{theorem:main} and let $u$ be the global free energy solution constructed in \cite{BCM}.
Consequently, thanks to Lemma \ref{lemma:mod} the solution admits a unique decomposition on some small time interval $[0,T^*)$ :
\begin{align}
u(x,t)=\frac{1}{R^2(t)\mu(t)}\Big(\tilde{Q}_\mu\Big(\frac{x}{R(t)\sqrt{\mu(t)}}\Big)+\varepsilon\Big(\frac{x}{R(t)\sqrt{\mu(t)}},t\Big)\Big),
\end{align}
where $$(\varepsilon(t),1)=(\varepsilon(t),|\cdot|^2)=0,$$
and from standard argument (see \cite{MM}) $\mu(t)$ satisfies $\mu\in C^1([0,T^*))$.
In all the rest of the manuscript we will make an abuse of notation and consider $\mu$ also as a function of $s$.
Using the initial smallness assumption on $\varepsilon_0$ and $\mu_0$, we prove the following Proposition:
\begin{proposition}\label{proposition:bootstrap}
Let $u_0$ be as in Theorem \ref{theorem:main} and let $u$ be the global free energy solution constructed in \cite{BCM}. Assume there exists $S^*$ such that for all $s\in[e,S^*)$, there exist $\eta>0$ and $\mu:=\mu(s)$ a $C^1$ function from $[e,S^*)$ to $\RR_+$, $\mu(e)=\mu_0$ and
\begin{align}\label{muhypboot}
\Big|\mu(s)-\frac{1}{2s}\Big|<\frac{A_1}{s|\log s|}.
\end{align}
Let us define $v(s,y)=\mu R^2 u(t,x)=\tilde{Q}_\mu(y)+\varepsilon(s,y)$ where $y=\frac{x}{R\sqrt{\mu}}$ and $\frac{ds}{dt}=\frac{1}{\mu R^2}$.
If we assume that for all $s\in[e,S^*)$
\begin{equation}\label{bootenergie}
\|\varepsilon\|_{L^2_{Q_\mu}}^2\leq A\mu,
\end{equation}
then the regime is trapped, in particular when $\frac{S^*}{2}<s\leq S^*$ we obtain the same inequalities \eqref{muhypboot} and \eqref{bootenergie} with $A_1$ and $A$ replaced by $\frac{A_1}{2}$ and $\frac{A}{2}$.
\end{proposition}
Hence, if one can increase by continuity the size of the intervall $[e,S^*)$, it will induce that we can also increase the size of $[0,T^*)$.
\section{Linear operator and spectral Gap}

In this section, we describe the properties of the linearized operator $\mathcal{L}_\mu^y$ obtained from the linearization around $Q_\mu$.
If we plug $v(y,s)=\tilde{Q}_\mu(y)+\varepsilon(y,s)$ in \eqref{v2rescal}  we get:
\begin{align}
 \partial_s\varepsilon +\partial_s \tilde{Q}_\mu &= -\nabla\cdot(\varepsilon\nabla\phi_{\varepsilon})+(\mu+\frac{\mu_s}{2\mu})(\Lambda \varepsilon+\Lambda \tilde{Q}_\mu)+\Delta \tilde{Q}_\mu \nonumber\\
 &+\Delta\varepsilon-\nabla\cdot(\tilde{Q}_\mu\nabla\phi_\varepsilon)-\nabla\cdot(\varepsilon\nabla\phi_{\tilde{Q}_\mu})-\nabla\cdot(\tilde{Q}_\mu\nabla\phi_{\tilde{Q}_\mu}).
 \end{align}
Hence,
\begin{align}\label{eqn:linQmu}
 \partial_s\varepsilon &={\mathcal{L}_\mu^y}\varepsilon+\frac{\mu_s}{2\mu}\Lambda\varepsilon+\Theta_\mu(\varepsilon)+N(\varepsilon)+\tilde{E},
 \end{align}
where
$${\mathcal{L}_\mu^y}\varepsilon =  \nabla\cdot\Big(Q_\mu\nabla({\mathcal{M}_\mu^y}\varepsilon)\Big)=\Delta\varepsilon+\mu\Lambda\varepsilon-\nabla\cdot( \varepsilon\nabla\phi_{Q_\mu}+Q_\mu\nabla\phi_\varepsilon),$$
the previous equality just comes from the fact that $\frac{\nabla Q_\mu}{Q_\mu}=\nabla\phi_{Q_\mu}-\mu y$ which we will prove in Proposition \ref{proposition:algebraic_identities}.
In addition,
$${\mathcal{M}_\mu^y}\varepsilon=\frac{\varepsilon}{Q_\mu}-\phi_{\varepsilon}.$$
Note that ${\mathcal{M}_\mu^y}$ corresponds to the linearisation arround $Q_\mu$ of the free energy
$$\mathcal{F}(u)=\int_{\RR^2}u\log udx-\frac{1}{2}\int_{\RR^2}u\phi_u dx.$$
The error term$\tilde{E}$ comes from the fact that $\tilde{Q_\mu}$ is not an exact solution of \eqref{v2rescal}. It is given by
\begin{align}\label{Error}
\tilde{E}=\Delta\tilde{Q}_\mu+\mu\Lambda\tilde{Q}_\mu-\nabla\cdot(\tilde{Q}_\mu\nabla\phi_{\tilde{Q}_\mu})+\frac{\mu_s}{2\mu}\Lambda \tilde{Q}_\mu -\partial_s \tilde{Q}_\mu.
\end{align}

We split it into 2 parts: 
$$\tilde{E}=E+F,$$
with 
$$E=\Delta\tilde{Q}_\mu+\mu\Lambda\tilde{Q}_\mu-\nabla\cdot(\tilde{Q}_\mu\nabla\phi_{\tilde{Q}_\mu}),\quad \mbox{ and } F=\frac{\mu_s}{2\mu}\Lambda \tilde{Q}_\mu -\partial_s \tilde{Q}_\mu.$$
The nonlinear term $N(\varepsilon)$ is given by,
\begin{equation}\label{nonlinear}
N(\varepsilon)=-\nabla\cdot(\varepsilon\nabla\phi_{\varepsilon}).
\end{equation}

The linear term $\Theta_\mu$ measures the error due to the fact that we are linearizing arround $\tilde{Q}_\mu$ and not arround $Q_\mu$. It is given by:
$$\Theta_\mu(\varepsilon)=\nabla\cdot[(\tilde{Q}_\mu-Q_\mu)\nabla\phi_\varepsilon+\varepsilon\nabla\phi_{\tilde{Q}_\mu-Q_\mu}].$$
We have the following relations between ${\mathcal{L}_\mu^y}$ and $\mathcal{L}^z$:
$${\mathcal{L}_\mu^y}(v(y))=\mu\mathcal{L}^z\Big(\frac{1}{\mu}v(\frac{z}{\sqrt{\mu}})\Big).$$
We now prove fundamental algebraic relations on $Q_\mu$, ${\mathcal{M}_\mu^y}$ and ${\mathcal{L}_\mu^y}$:
\begin{proposition}\label{proposition:algebraic_identities}
Let $Q_\mu$ be the solution of \eqref{eq_Q_mu},
then $\phi_{Q_\mu}$ solves:
\begin{align}\label{eqphiqmu}
-\Delta\phi_{Q_\mu}=-\phi_{Q_\mu}''-\frac{1}{r}\phi_{Q_\mu}'=Q_\mu=8e^{\phi_{Q_\mu}-\mu\frac{|y|^2}{2}},
\end{align}
\begin{itemize}
\item \textbf{Algebraic identities for ${\mathcal{M}_\mu^y}$}
$${\mathcal{M}_\mu^y}(\partial_\mu Q_\mu)=-\frac{|y|^2}{2},$$
$${\mathcal{M}_\mu^y}(\partial_i Q_\mu)=-\mu y_i,$$
$${\mathcal{M}_\mu^y}(\Lambda Q_\mu)=2-\frac{M}{2\pi}-\mu|y|^2,$$ 
$$M'(\mu) \mu\partial_M n_\infty(\sqrt{\mu}y)=\partial_\mu Q_\mu-\frac{1}{2\mu}\Lambda Q_\mu.$$
$${\mathcal{M}_\mu^y}\bigl(\partial_\mu Q_\mu-\frac{1}{2\mu}\Lambda Q_\mu\bigr)=-\Big(2-\frac{M}{2\pi}\Big).$$
\item \textbf{Algebraic identities for ${\mathcal{L}_\mu^y}$}
$$\hspace{0.5cm}{\mathcal{L}_\mu^y}(\partial_\mu Q_\mu)=-\Lambda(Q_\mu), $$       
$$\hspace{0.5cm}{\mathcal{L}_\mu^y}(\partial_i Q_\mu)=-\mu \partial_i Q_\mu,$$
\hspace{0.5cm}$${\mathcal{L}_\mu^y}(\Lambda Q_\mu)=-2\mu\Lambda Q_\mu$$
$${\mathcal{L}_\mu^y}\bigl(\partial_\mu Q_\mu-\frac{1}{2\mu}\Lambda Q_\mu\bigr)=0.$$
\end{itemize}
\end{proposition}

\begin{proof}[Proof of Proposition \ref{proposition:algebraic_identities}]
\begin{itemize}
\item \underline{Derivation of \eqref{eqphiqmu}.}\\
Let us recall the properties of $n_\infty^M(z)$ for $M<8\pi$:
$$-\Delta \phi_{n_\infty^M}= n_\infty^M = M \frac{e^{\phi_{n_\infty^M}-\frac{|z|^2}{2}}}{\int_{\RR^2}e^{\phi_{n_\infty^M}-\frac{|z|^2}{2}}dz},$$
with $M=\int_{\RR^2}n_\infty^M dz<8\pi$ and $$n_\infty^M(0)=M\frac{e^{\phi_{n_\infty^M}(0)}}{\int_{\RR^2}e^{\phi_{n_\infty^M}-\frac{|z|^2}{2}}dz}$$ or $$\phi_{n_\infty^M}(0)=\log(n_\infty^M(0)\int_{\RR^2}e^{\phi_{n_\infty^M}-\frac{|z|^2}{2}}dz)-\log(M).$$
Hence if we set $\tilde{\phi}_{n_\infty^M}=\alpha +\phi_{n_\infty^M}$ where $\alpha =\log(M)-\log(\int_{\RR^2}e^{\phi_{n_\infty^M}-\frac{|z|^2}{2}}dz)$, then we get 
$$-\Delta\tilde{\phi}_{n_\infty^M}=n_\infty^M(z)=e^{\tilde{\phi}_{n_\infty^M}-\frac{|z|^2}{2}},$$
with $\tilde{\phi}_{n_\infty^M}(0)=\alpha +\phi_{n_\infty^M}(0)=\log(n_\infty^M(0))$.
Now we see that if we choose $\mu=\frac{8}{n_\infty^M(0)}$, $y=\frac{z}{\sqrt{\mu}}$, and 
$$\phi_{Q_\mu}(y)=\tilde{\phi}_{n_\infty^M}(\sqrt{\mu} y)-\log(n_\infty^M(0)),$$
then, $\phi_{Q_\mu}$ solves \eqref{eqphiqmu} with $\phi_{Q_\mu}(0)=0$ and $\phi_{Q_\mu}'(0)=0$.
 \item \underline{Differentiation of \eqref{eqphiqmu} with respect to $y$.}\\
Moreover, we have  by differentiating \eqref{logeqmu} in $y$
 \begin{align}\label{logeqmu}
 \log(Q_\mu)=\log(8)+\phi_{Q_\mu}-\mu\frac{|y|^2}{2}:
 \end{align}
$$\frac{\partial_i Q_\mu}{Q_\mu}=\phi_{\partial_i Q_\mu}-\mu y_i,$$
then 
$${\mathcal{M}_\mu^y}(\partial_i Q_\mu)=-\mu y_i.$$
Hence,
$${\mathcal{L}_\mu^y}(\partial_i Q_\mu)=-\mu \partial_i Q_\mu.$$
\item \underline{Differentiation of \eqref{eqphiqmu} with respect to $\mu$.}\\
We see easily by differentiating \eqref{logeqmu} in $\mu$ that\begin{align}
\frac{\partial_\mu Q_\mu}{Q_\mu}=\phi_{\partial_\mu Q_\mu}-\frac{|y|^2}{2},
\end{align}
which means 
$${\mathcal{M}_\mu^y}(\partial_\mu Q_\mu)=-\frac{|y|^2}{2}.$$
It follows that
$${\mathcal{L}_\mu^y}(\partial_\mu Q_\mu)=-\nabla\cdot(yQ_\mu).$$
Actually, we differentiated $Q_\mu$ with respect to $\mu$ without justifying that $Q_\mu$ is smooth with respect to $\mu$.
However, this will be done later in section 4 (see \eqref{gronwallineq}).
\item \underline{Differentiation of \eqref{eqphiqmu} with respect to $\lambda$.}\\
In addtition, if we set $Q^\lambda_\mu(y)=\lambda^2Q_\mu(\lambda y)$ where $\lambda>0$, then 
$$\phi_{Q_\mu}(\lambda y)=-\frac{M(\mu)}{2\pi}\log(\lambda)+\phi_{Q^\lambda_\mu}(y).$$
Using \eqref{logeqmu}, we also get
$$\log(Q^\lambda_\mu)=\log(8)+\phi_{Q^\lambda_\mu}+(2-\frac{M}{2\pi})\log(\lambda)-\mu\lambda^2\frac{|y|^2}{2}.$$
If we differentiate the previous equation in $\lambda$ and set $\lambda=1$, we deduce
\begin{align}\label{Mlambdaqmu}
\frac{\Lambda Q_\mu}{Q_\mu}-\phi_{\Lambda Q_\mu}={\mathcal{M}_\mu^y}(\Lambda Q_\mu)=2-\frac{M}{2\pi}-\mu|y|^2.
\end{align}
It follows 
$${\mathcal{L}_\mu^y}(\Lambda Q_\mu)=-2\mu\Lambda Q_\mu.$$ 
The relation $M'(\mu) \mu\partial_M n^M_\infty(\sqrt{\mu}y)=\partial_\mu Q_\mu-\frac{1}{2\mu}\Lambda Q_\mu$ is a consequence of an easy computation.  From this relation and ${\mathcal{L}_\mu^y}(\partial_\mu Q_\mu)=-\Lambda(Q_\mu)$, \hspace{0,1cm}  ${\mathcal{L}_\mu^y}(\Lambda Q_\mu)=-2\mu\Lambda Q_\mu$, we deduce that ${\mathcal{L}_\mu^y}(\partial_\mu Q_\mu-\frac{1}{2\mu}\Lambda Q_\mu)=0$. One can also recover this from the fact that 
$${\mathcal{L}^z}(\partial_M n^M_\infty)=0\Rightarrow {\mathcal{L}_\mu^y}(\partial_\mu Q_\mu-\frac{1}{2\mu}\Lambda Q_\mu)=0.$$
 \end{itemize}
 \end{proof}
\begin{lemma}\label{lemma:Mprop}
The operator $Q_\mu{\mathcal{M}_\mu^y}:L^2_{Q_\mu}\longrightarrow L^2_{Q_\mu}$, given by
$$Q_\mu{\mathcal{M}_\mu^y} u=u-{Q_\mu}\phi_u,$$
is linear continuous, self adjoint and Fredholm.
Moreover,
$$\int_{\RR^2}Q_\mu|{\mathcal{M}_\mu^y} u|^2\leq  K^*\| u\|_{L^2_{Q_\mu}}^2,$$
with $K^*$ uniform with respect to $\mu$,
$$\forall (u,v)\in L^2_{Q_\mu}\times L^2_{Q_\mu}, ({\mathcal{M}_\mu^y} u,v)=(u,{\mathcal{M}_\mu^y} v).$$
In addition, 
 $Q_\mu{\mathcal{M}_\mu^y}:L^2_{Q_\mu,0}\longrightarrow L^2_{Q_\mu,0}$
is positive definite, and defines an inner product on $L^2_{Q_\mu,0}$.
\end{lemma}
\begin{proof}
We use the following inequality from \cite{RS}
$$\|\phi_u\|_{L^\infty(\{r\leq1\})}+\Big\|\frac{\phi_u}{1+|\log<r>|}\Big\|_{L^\infty(\{r\geq1\})}\lesssim \|u\|_{L^2_Q}.$$
Hence,
\begin{eqnarray}
\int_{\RR^2}Q_\mu|{\mathcal{M}_\mu^y} u|^2dy\lesssim \|u\|^2_{L^2_{Q_\mu}}+\Big\|\frac{\phi_u}{1+|\log |y||}\Big\|^2_{L^\infty}\int_{\RR^2}(1+|\log(|y|)|)Q_\mu dy\lesssim \|u\|^2_{L^2_{Q_\mu}}.
\end{eqnarray}
The self adjointness follows from the fact that
$$\mbox{ for all } (u,v)\in L^2_{Q_\mu}\times L^2_{Q_\mu}, (\phi_u,v)_{L^2}=(u,\phi_v)_{L^2}.$$
It is easy to see that the operator is in the form $I-Q_\mu\phi_{(\cdot)}$ with $Q_\mu\phi_{(\cdot)}$ a compact operator which implies that it is Fredholm. 
To prove that $(\mathcal{M}_\mu^y v,v)\geq0$ for any $v\in L^2_{Q_\mu,0}$ , we recall that the free energy $\mathcal{F}(v)$ achieves its minimum at $v = Q_\mu$ according to the Logarithmic Hardy-Littlewood-Sobolev inequality and observe that
\begin{align}\label{freemini}
(\mathcal{M}_\mu^y v,v)=\lim_{\delta\rightarrow0}\frac{1}{\delta^2}\mathcal{F}(Q_\mu+\delta v)\geq0,
\end{align}
for any smooth function $v$ compactly supported and satisfying $\int_{\RR^2} v dy=0$. Hence, by density \eqref{freemini} holds for any $v$ in $L^2_{Q_\mu,0}(\RR^2)$.
\end{proof}
Let's define the domain of ${\mathcal{L}_\mu^y}$,
$$D({\mathcal{L}_\mu^y})=\{f\in L^2_{{Q_\mu}} \mbox{ such that }\mathcal{L}_\mu^y(f)\in L^2_{{Q_\mu}}\}\subset L^2_{{Q_\mu}}.$$
\begin{proposition} \label{proposition:propofL}
${\mathcal{L}_\mu^y}$ is a self-adjoint operator and has discrete spectrum on $D_0({\mathcal{L}_\mu^y})\subset L^2_{{Q_\mu},0}$ for the inner product $(\cdot,\cdot)_{{\mathcal{M}_\mu^y}}$, where 
$$D_0({\mathcal{L}_\mu^y})=\{f\in L^2_{{Q_\mu},0} \mbox{ such that }\mathcal{L}_\mu^y(f)\in L^2_{{Q_\mu},0}\}.$$
Moreover in $D({\mathcal{L}_\mu^y})$, 
\begin{align}\label{kernelLmuy}
Ker({\mathcal{L}_\mu^y})&=Span(\partial_\mu Q_\mu-\frac{1}{2\mu}\Lambda Q_\mu),\\
\label{eigen1}
{\mathcal{L}_\mu^y}(\partial_i Q_\mu)&=-\mu \partial_i Q_\mu,\\
\label{eigen2}
{\mathcal{L}_\mu^y}(\Lambda Q_\mu)&=-2\mu\Lambda Q_\mu.
\end{align}
\end{proposition}
Furthermore, Proposition \ref{proposition:propofL} was proved in \cite{CD}(using the $z$ coordinate).
Campos and Dolbeault obtained the following spectral gap inequality \cite{CD}:
\begin{theorem}\label{theorem:CD} 
For all $w\in L^2_{n_\infty^M}(\RR^2)$.\\
If 
\begin{equation}\label{spectcondz}
(w,\mathcal{M}^z\partial_M n_\infty^M)_{L^2}=C_1(M)(w,1)=0
\end{equation}
then 
\begin{equation}\label{spectralgap1z}
 (\mathcal{M}^z w,w)\leq -(\mathcal{L}^z w,\mathcal{M}^z w). 
 \end{equation}
There exists $K_2>2$ such that if in addition $$(w,\mathcal{M}^z(\partial_1 n_\infty^M))_{L^2}=(w,\mathcal{M}^z(\partial_2 n_\infty^M))_{L^2}=(w,\mathcal{M}^z\Lambda n_\infty^M)_{L^2}=0$$ then
\begin{equation}\label{spectralgap2z}
  K_2(\mathcal{M}^z w,w)\leq -(\mathcal{L}^z w,\mathcal{M}^z w).
 \end{equation}
\end{theorem}
In the radial case, \eqref{spectralgap2z} follows from the study of the spectrum of $\mathcal{L}^z$ performed in \cite{CD}. Moroever, the numerics performed in \cite{CD} (see Figure 1 in \cite{CD}) show that it is also true in the nonradial case. We hope to address this elsewhere.
We rewrite Theorem \ref{theorem:CD} in the $y$ coordinate with $y=\frac{z}{\sqrt{\mu}}$.
\begin{corollary}\label{corollary:NT} 
For all $w\in L^2_{Q_\mu}(\RR^2)$.\\
If 
\begin{equation}\label{spectcondy}
(w,{\mathcal{M}_\mu^y}(\partial_\mu Q_\mu-\frac{1}{2\mu}\Lambda Q_\mu))=C_2(\mu)(w,1)=0\end{equation}
then 
\begin{equation}\label{spectralgap1y}
 \mu({\mathcal{M}_\mu^y} w,w)\leq -({\mathcal{L}_\mu^y} w,{\mathcal{M}_\mu^y} w). 
 \end{equation}
There exists $K_2>2$ such that if in addition $$(w,{\mathcal{M}_\mu^y}(\partial_1 Q_\mu))_{L^2}=(w,{\mathcal{M}_\mu^y}(\partial_2 Q_\mu))_{L^2}=(w,{\mathcal{M}_\mu^y}\Lambda Q_\mu)_{L^2}=0$$ then   
\begin{equation}\label{spectralgap2y}
K_2\mu({\mathcal{M}_\mu^y} w,w)\leq -({\mathcal{L}_\mu^y} w,{\mathcal{M}_\mu^y} w).
 \end{equation}
\end{corollary}
\begin{remark}
Notice here that the conditions $(w,1)_{L^2}=(w,|\cdot|^2)_{L^2}=0$ correspond in the inner product $(\cdot,\cdot)_{\mathcal{M}_\mu^y}$ where ${\mathcal{L}_\mu^y}$ is self-adjoint to the orthogonality on $\Lambda Q_\mu$ and on $\partial_\mu Q_\mu-\frac{1}{2\mu}\Lambda Q_\mu$ the kernel of ${\mathcal{L}_\mu^y}$.
Indeed,
$$(w,\mathcal{M}_\mu^y\Lambda Q_\mu)=(2-\frac{M}{2\pi})(w,1)-\mu(w,|\cdot|^2)=0,$$
and
$$(w,\mathcal{M}_\mu^y\partial_\mu Q_\mu)=-\frac{1}{2}(w,|\cdot|^2)=0.$$
In addition, we get for free from the conservation of the center of mass that $(w,y_i)_{L^2}=0$ which corresponds to the orthogonality on $\partial_i Q_\mu$ in the inner product $(\cdot,\cdot)_{\mathcal{M}_\mu^y}$.
Indeed,
$$(w,\mathcal{M}_\mu^y\partial_i Q_\mu)=-\mu\int_{\RR^2}w y_i dy=0.$$
\end{remark}

\section{Construction of the approximate profile}

In this section, we construct an adequat approximate profile $\tilde{Q}_\mu$ that has finite second moment and $8\pi$ mass.
Recall that $\mu$ and $M$ are related by $\mu=\frac{8}{n_\infty^M(0)}$.
\begin{proposition}\label{proposition:comp1}
Let $n_\infty^M$ be the solution of \eqref{ninfiny}, then $\mu n_\infty^M(\sqrt{\mu} y)\rightarrow Q(y)$ uniformly as $\mu\rightarrow 0$. Moreover,
\begin{align}\label{phiq_qmu}
\phi_Q-\mu\frac{|y|^2}{2}<\phi_{Q_\mu}(y)-\frac{\mu |y|^2}{2}<\phi_Q(y)<\phi_{Q_\mu}(y)<0 \mbox{ for all }y\neq0,
\end{align}
\begin{align}\label{phi'q_qmu}
y\cdot\nabla\phi_{Q_\mu}(y)-\mu |y|^2<y\cdot\nabla\phi_Q(y)<y\cdot\nabla\phi_{Q_\mu}(y)<0 \mbox{ for all }y\neq0,
\end{align}
\begin{align}\label{Q_Qmu}
Q(y)e^{-\frac{\mu|y|^2}{2}}< \mu n_\infty^M(\sqrt{\mu} y)=Q_\mu(y)< Q(y), \mbox{ for all } y\in\RR^2.
\end{align}
\end{proposition}

\begin{proof}[Proof of Proposition \ref{proposition:comp1}]
We set $r=|y|$ and recall that $\phi_Q$ is characterized by
\begin{equation}\label{phiQ1}
-\phi_Q''-\frac{1}{r}\phi_Q'=8e^{\phi_Q}=Q \mbox{ with } \phi_Q(0)=0 \mbox{ and } \phi_Q'(0)=0.
\end{equation}
The unique solution of \eqref{phiQ1} is given by $\phi_Q(y)=-2\log(1+|y|^2)$.
From Proposition \ref{proposition:algebraic_identities}, we have also
$$-\Delta\phi_{Q_\mu}=8e^{\phi_{Q_\mu}}e^{-\frac{\mu r^2}{2}}=Q_\mu.$$
We will prove that
$$\phi_Q-\frac{\mu r^2}{2}< \phi_{Q_\mu}-\frac{\mu r^2}{2}<\phi_Q<\phi_{Q_\mu}<0, \quad\forall r\neq 0.$$
After taking the exponantial, we get that 
$$Qe^{-\frac{\mu |y|^2}{2}}< Q_\mu(y)< Q(y).$$
Indeed,
\begin{eqnarray}\label{eqn:compar1}
\left\{
\begin{array}{ll}
\Delta\phi_Q+8e^{\phi_Q} =0 \\
\Delta\phi_{Q_\mu}+8e^{\phi_{Q_\mu}}e^{-\frac{\mu r^2}{2}}=0
\end{array}
\right.
\end{eqnarray}
with $\phi_Q(0)=\phi_{Q_\mu}(0)=0=\phi'_Q(0)=\phi_{Q_\mu}'(0)$. 
Since $\phi_{Q_\mu}(0)=0=\phi_Q(0)$ then from 
\begin{eqnarray}
\left\{
\begin{array}{ll}
\phi_Q(r)=-8\int_0^r\frac{1}{r'}\int_0^{r'} \tau e^{\phi_Q(\tau)}d\tau dr' \\
\phi_{Q_\mu}(r)=-8\int_0^r\frac{1}{r'}\int_0^{r'} \tau e^{\psi_\mu(\tau)}d\tau dr',
\end{array}
\right.
\end{eqnarray}
we deduce easily that
 $\phi_Q(y)<0$,\quad $\phi'_Q(y)<0$ and  $\phi_{Q_\mu}(y)<0$,\quad $\phi'_{Q_\mu}(y)<0$.
It follows that $\phi_{Q_\mu}$ and $\phi_Q$ are radially symetric and decreasing.
Set $\psi_\mu(y)=\phi_{Q_\mu}-\frac{\mu|y|^2}{2}$.
\begin{itemize}
\item{\textbf{Proof of $\phi_Q-\frac{\mu r^2}{2}< \phi_{Q_\mu}-\frac{\mu r^2}{2}<\phi_Q<\phi_{Q_\mu}<0$  for $r$ in a neighborhood of 0.}}\\
We first expand $\phi''_{Q_\mu},\phi'_{Q_\mu},\phi_Q,\phi_{Q_\mu},\psi_\mu$ around 0:
$$\phi_Q(r)=-2r^2+r^4+o(r^4),$$
$$\phi''_{Q_\mu}(r)=\phi''_{Q_\mu}(0)+r\phi^{(3)}_{Q_\mu}(0)+\frac{r^2}{2}\phi^{(4)}_{Q_\mu}(0)+O(r^3),$$
$$\phi'_{Q_\mu}(r)=r\phi''_{Q_\mu}(0)+\frac{r^2}{2}\phi^{(3)}_{Q_\mu}(0)+\frac{r^3}{6}\phi^{(4)}_{Q_\mu}(0)+O(r^4),$$
$$\phi_{Q_\mu}(r)=\sum_{n=2}^{4}\frac{r^n}{n!}\phi_{Q_\mu}^{(n)}(0)+o(r^4),$$
and
$$\psi_\mu(r)=\sum_{n=0}^{4}\frac{r^n}{n!}\psi_\mu^{(n)}(0)+o(r^4).$$
Then, to compare the 3 functions we just need to find the values of the consecutive derivatives of $\phi_{Q_\mu}$ at 0. To do so, we plug the expansions of 
 $\phi''_{Q_\mu},\phi'_{Q_\mu}$ and $\phi_{Q_\mu}$ arround zero in 
 $$-\phi_{Q_\mu}''-\frac{1}{r}\phi_{Q_\mu}'=8e^{\phi_{Q_\mu}-\mu\frac{r^2}{2}}$$ we get,
$$\phi_{Q_\mu}''(0)=-4,$$
$$\phi^{(3)}_{Q_\mu}(0)=0,$$
and
$$\phi^{(4)}_{Q_\mu}(0)=6(4+\mu).$$
The conclusion follows by substituing the values of the derivatives of $\phi_{Q_\mu}$ in the expansion of $\phi_{Q_\mu}$ and $\psi_\mu$ arround 0.
\item{\textbf{Proof of $\phi_Q-\frac{\mu r^2}{2}< \phi_{Q_\mu}-\frac{\mu r^2}{2}<\phi_Q<\phi_{Q_\mu}<0$, for all $r>0$.}}\\
Let's argue by contradiction, suppose there exists $r_0>0$ such that $\phi_Q(r_0)=\phi_{Q_\mu}(r_0)$ and $\phi_{Q_\mu}(r)>\phi_Q(r)>\psi_\mu(r)$ for all $0<r<r_0$. 
We use that $\phi_Q$ and $\phi_{Q_\mu}$ are solutions of
\begin{eqnarray}\label{eqn:compar1}
\left\{
\begin{array}{ll}
(r\phi_Q')'+8re^{\phi_Q} =0 \\
(r\phi_{Q_\mu}')'+8re^{\phi_{Q_\mu}-\frac{\mu r^2}{2}}=0,
\end{array}
\right.
\end{eqnarray}
which can be solved, using the conditions at 0:
\begin{eqnarray}\label{eqn:comparint1}
\left\{
\begin{array}{ll}
\phi_Q(r)=-8\int_0^r\frac{1}{r'}\int_0^{r'} \tau e^{\phi_Q(\tau)}d\tau dr' \\
\phi_{Q_\mu}(r)=-8\int_0^r\frac{1}{r'}\int_0^{r'} \tau e^{\psi_\mu(\tau)}d\tau dr'.
\end{array}
\right.
\end{eqnarray}
It follows that
$$\phi_Q(r)-\phi_{Q_\mu}(r)=-8\int_0^r\frac{1}{r'}\int_0^{r'} \tau[e^{\phi_Q(\tau)}-e^{\psi_\mu(\tau)}]d\tau dr',$$
and since $e^{\phi_Q}-e^{\psi_\mu}>0$ for $r\in(0,r_0)$, we get
$$\phi_{Q}(r_0)-\phi_{Q_\mu}(r_0)=0=-8\int_0^{r_0}\frac{1}{r'}\int_0^{r'} \tau[e^{\phi_Q(\tau)}-e^{\psi_\mu(\tau)}]d\tau dr'<0,$$
which is a contradiction.
Now suppose there exists $r_0$ such that $\phi_Q(r_0)=\psi_\mu(r_0)$ and $\phi_{Q_\mu}(r)>\phi_Q(r)>\psi_\mu(r)$ for all $0<r<r_0$.
Since,
$$\phi_{Q_\mu}(r)=-8\int_0^r\frac{1}{r'}\int_0^{r'} \tau e^{\psi_\mu(\tau)}d\tau dr',$$
and $\psi_\mu=\phi_{Q_\mu}-\mu\frac{r^2}{2}$, it follows
$$\psi_{\mu}(r)=-8\int_0^r\frac{1}{r'}\int_0^{r'} \tau[e^{\psi_\mu(\tau)}]d\tau dr'-\mu\frac{r^2}{2}.$$
Then, using that $\phi_Q(r)<\phi_{Q_\mu}(r)$, we get
$$\psi_\mu(r_0)<\underbrace{-8\int_0^{r_0}\frac{1}{r'}\int_0^{r'} \tau[e^{\phi_Q(\tau)-\mu\frac{\tau^2}{2}}]d\tau dr'-\mu\frac{r_0^2}{2}}_{\tilde{\psi}_\mu(r_0)}.$$
 Notice that for $\mu=0$ we have $\tilde{\psi}_0(r_0)=\phi_Q(r_0)$.
If we prove that ${\partial_\mu \tilde{\psi}_\mu}(r_0)<0$, we deduce that $\psi_\mu(r_0)<\tilde{\psi}_0(r_0)=\phi_Q(r_0)$ which is the desired contradiction. Since,
$${\partial_r\partial_\mu \tilde{\psi}_\mu}(r_0)=\frac{4}{r_0}\int_0^{r_0}\tau^3[e^{\phi_Q(\tau)-\mu\frac{\tau^2}{2}}]d\tau-r_0<\partial_r{\partial_\mu \tilde{\psi}_\mu}_{|_{\mu=0}}=\frac{4}{r_0}\int_0^{r_0}\tau^3e^{\phi_Q}d\tau-r_0,$$ then it suffices to prove that $\partial_{r}{\partial_\mu \tilde{\psi}_\mu}_{|_{\mu=0}}<0.$
Indeed, one can easily calculate
$$\partial_{r}{\partial_\mu \tilde{\psi}_\mu}_{|_{\mu=0}}(r_0)=\frac{2}{r_0}\Big(-1+\frac{1}{1+r_0^2}+\log(1+r_0^2)\Big)-r_0<0, \mbox{ for all } r_0>0,$$
which concludes the proof.
\end{itemize}
The proof of \eqref{phi'q_qmu} is similar to \eqref{phiq_qmu}, and is left to the reader. Note that \eqref{phi'q_qmu} implies \eqref{phiq_qmu}.
\end{proof}
For the rest of our estimations we will need the following proposition. Recall that $T_1$ was also used in \cite{RS}.
\begin{proposition}\label{proposition:comp}
Let $T_1$ be the solution of $\mathcal{L}T_1=\nabla\cdot(Q\nabla\mathcal{M}T_1)=\Lambda Q$, then we have 
$$\phi_{T_1}(y)=O(r^4\log <r>)\indic_{\{ r\leq 1\}}+\Big[6(\log r)^2+O\Big(\frac{(\log r)^2}{r^2}\Big)\Big]\indic_{\{ r\geq 1\}}.$$
$$\phi_{T_1}'(r)=O(r^3\log <r>)\indic_{\{ r\leq 1\}}+O\Big(\frac{\log r}{r}\Big)\indic_{\{ r\geq 1\}}.$$
\end{proposition}

\begin{proof}
First we notice that $\nabla\cdot(Q\nabla\mathcal{M}T_1)=\nabla\cdot(y Q)$, with $T_1(0)=0$ and $\nabla T_1(0)=0$ which implies $\mathcal{M}T_1=\frac{|y|^2}{2}$.
And since $T_1$ is radial and $-\Delta \phi_{T_1}= T_1$ we deduce that $\phi_{T_1}$ is solution of
$$L\phi_{T_1}=-\phi_{T_1}''-\frac{1}{r}\phi_{T_1}'-Q\phi_{T_1}=Q\frac{r^2}{2},$$
with $\phi_{T_1}(0)=0,\phi_{T_1}'(0)=0$.
To estimate $\phi_{T_1}$ we need to invert the operator $L$.
Indeed, the Green's function of $L$ is explicit and the set of radial solutions to the homogeneous problem
$$L f=0$$
is spanned by 
$$f_0=1-\frac{2}{1+r^2},\quad f_1=\frac{r^2\log(r)-2-\log(r)}{1+r^2},$$
with Wronskian 
$$W=f'_1f_0-f_1f'_0=\frac{1}{r}.$$
Hence a solution of 
$$Lf=g, \mbox{ with } f(0)=f'(0)=0$$
is given by
\begin{align}\label{invop}
f(r)=-f_0(r)\int_0^rgf_1\tau d\tau+f_1(r)\int_0^r gf_0\tau d\tau.
\end{align}
It follows that 
$$\phi_{T_1}(r)=-\frac{f_0(r)}{2}\int_0^rQf_1(\tau)\tau^3d\tau+\frac{f_1(r)}{2}\int_0^rf_0(\tau)Q\tau^3d\tau.$$
Hence,
$$\phi_{T_1}(y)=O(r^4\log <r>)\indic_{\{ r\leq 1\}}+\Big[6(\log r)^2+O\Big(\frac{(\log r)^2}{r^2}\Big)\Big]\indic_{\{ r\geq 1\}}.$$
Now to estimate $\phi_{T_1}'$ we use that
$$\phi_{T_1}'(r)=-\frac{f'_0(r)}{2}\int_0^rQf_1(\tau)\tau^3d\tau+\frac{f'_1(r)}{2}\int_0^rf_0(\tau)Q\tau^3d\tau,$$
with
$$f_1'(r)=\frac{4r^2\log r+r^4+4r^2-1}{8r}Q \mbox{ and } f'_0(r)=\frac{1}{2}rQ.$$
We finally deduce that
$$\phi_{T_1}'(r)=O(r^3\log <r>)\indic_{\{ r\leq 1\}}+O(\frac{\log r}{r})\indic_{\{ r\geq 1\}},$$
which concludes the proof.
\end{proof}
Now we write $Q_\mu$ explicitely in terms of $Q$:
\begin{proposition}\label{proposition:approxQmu}
There exists $\sigma:=\sigma(\mu,r)$  such that 
\begin{eqnarray}\label{approxphiqmu}
\phi_{Q_\mu}=\phi_Q+\mu\phi_{T_1}+\sigma 
\end{eqnarray}
where $\sigma$ satisfies
\begin{align}\label{sigma}
|\sigma(\mu,r)|\lesssim\min\Big( \mu^2r^2|\log <r>|,\mu(\log <r>)^2\Big).
\end{align} 
Moreover,
\begin{eqnarray}\label{approxgradphiqmu}
\phi'_{Q_\mu}=\phi'_Q+\mu\phi'_{T_1}+\partial_r\sigma, 
\end{eqnarray}
with
\begin{align}\label{drsigma}
|\partial_r\sigma|\lesssim\min\Big(\mu^2r|\log <r>|,\mu\frac{|\log <r>|}{r}\Big).
\end{align} 
In addition, we obtain the following expansion of $Q_\mu$:
\begin{eqnarray}\label{approxQmu}
Q_\mu=Qe^{\mu\phi_{T_1}-\frac{\mu r^2}{2}+\sigma(\mu,r)}.
\end{eqnarray}
\end{proposition}

\begin{proof}
We write \eqref{phiq_qmu} in a more precise way.
We introduce $\psi_1:=\psi_1(r)$ and $\sigma:=\sigma(r,\mu)$ such that
$$\phi_{Q_\mu}=\phi_Q+ \mu \psi_1(r)+\sigma(\mu,r).$$
Hence, from \eqref{eqphiqmu} we have :
$$-\Delta\phi_Q-\mu \Delta\psi_1-\Delta \sigma=Qe^{\mu\psi_1-\frac{\mu r^2}{2}+\sigma(\mu,r)}.$$
By \quad $0\leq \mu \psi_1+\sigma(\mu,y)\leq \mu\frac{r^2}{2}$,
we get that\quad $e^{\mu\psi_1+\sigma-\frac{\mu r^2}{2}}=1+\mu\psi_1-\mu\frac{r^2}{2}+\sigma(\mu,y)+O\Big(\min\Big((\mu^2 r^4),(\mu\frac{r^2}{2})\Big)\Big)$.
Hence,
$$-\mu\Big[ (\Delta+Q)\psi_1-Q\frac{r^2}{2}\Big]-(\Delta+Q)\sigma=
O\Big(\min\Big((\mu^2 r^4),(\mu\frac{r^2}{2})\Big)\Big)Q:=g(\mu,r).$$
If we select $\psi_1$ such that $(\Delta+Q)\psi_1-Q\frac{r^2}{2}=0$, then $\psi_1=\phi_{T_1}$ and we get that
$$-(\Delta+Q)\sigma=L\sigma=g(\mu,r).$$
From \eqref{invop}, we obtain that
\begin{align}\label{sigma2}
\sigma(\mu,r)=-f_0(r)\int_0^rgf_1\tau d\tau+f_1(r)\int_0^r gf_0\tau d\tau.
\end{align}
Hence,
$$\sigma(\mu,r)\lesssim\min\Big( \mu^2r^2|\log <r>|,\mu(\log r)^2\Big).$$
To conclude, \eqref{drsigma} follows easily by differentiating \eqref{sigma2} with respect to $r$.
\end{proof}
We derive in the following Proposition bounds on $\partial_\mu Q_\mu$ and $\phi_{\partial_\mu Q_\mu}$.
\begin{proposition}\label{proposition:comp2}
Let $Q_\mu$ be the solution of \eqref{eq_Q_mu},
then
\begin{align}\label{phiqdmu}
\phi_{\partial_\mu Q_\mu}=
\phi_{T_1}+\partial_\mu \sigma
\end{align}
where 
\begin{align}\label{dmusigma}
|\partial_\mu\sigma|&\lesssim \min\Big(\mu r^2|\log<r>|,|\log<r>|^2\Big).
\end{align}
Moreover,
\begin{align}\label{phidmuq}
\phi'_{\partial_\mu Q_\mu}=\phi'_{T_1}+\partial_r\partial_\mu \sigma,
\end{align}
with
$$|\partial_r\partial_\mu\sigma|\lesssim \min(\mu r|\log<r>|,\frac{|\log<r>|}{r}).$$
In addition, 
\begin{align}\label{dmuqmu}
\partial_\mu Q_\mu=(\phi_{T_1}+\partial_\mu\sigma-\frac{r^2}{2})Q_\mu.
\end{align}
\end{proposition}

\begin{proof}[Proof of Proposition \ref{proposition:comp2}]

By taking the derivative in $\mu$ of \eqref{sigma2}, we get
\begin{align}\label{difest}
\partial_\mu\sigma&=-f_0(r)\int_0^r(\phi_{T_1}+\partial_\mu\sigma-\frac{r^2}{2})(Q_\mu-Q)f_1\tau d\tau\nonumber\\
&\quad+f_1(r)\int_0^r (\phi_{T_1}+\partial_\mu\sigma-\frac{r^2}{2})(Q_\mu-Q)f_0\tau d\tau.
\end{align}
Since, $|f_0|\leq 1$, $|f_1|\leq |\log<r>|$ and $|Q_\mu-Q|\lesssim \min(\mu r^2,1) Q$, it follows that
$$|\partial_\mu\sigma|\lesssim \min(\mu r^2,|\log<r>|)|\log<r>| +|\log<r>|\int_0^r|\partial_\mu\sigma| \min(\mu\tau^2,1)\tau Q d\tau.$$
Hence, by using Gronwall inequality, we deduce 
\begin{align}\label{gronwallineq}
|\partial_\mu\sigma|\lesssim \min(\mu r^2,|\log<r>|)|\log<r>|. 
\end{align}
Note that to prove the smoothness in $\mu$ rigourosly one should do the previous calculation on the finite difference $\frac{\sigma(\mu)-\sigma(\mu_0)}{\mu-\mu_0}$ and then pass to the limit $\mu\rightarrow\mu_0$ in \eqref{gronwallineq}. 
We will not detail this here. To prove the bound on $\partial_r\partial_\mu\sigma$, we differentiate \eqref{difest} with respect to $r$,
\begin{align*}
\partial_r\partial_\mu\sigma&=-f'_0(r)\int_0^r(\phi_{T_1}+\partial_\mu\sigma-\frac{r^2}{2})(Q_\mu-Q)f_1\tau d\tau\\
&+f'_1(r)\int_0^r(\phi_{T_1}+\partial_\mu\sigma-\frac{r^2}{2})(Q_\mu-Q)f_0\tau d\tau.
\end{align*}
If we use as before that $|f_0|\leq 1$,\quad $f'_0(r)=\frac{rQ}{2}$,\quad $|f_1|\lesssim |\log<r>|$, \quad $|f'_1|\lesssim r^3Q$,\quad $|Q_\mu-Q|\lesssim \min(\mu r^2,1) Q$ and $|\partial_\mu\sigma|\lesssim \min(\mu r^2,|\log<r>|)|\log<r>|$, we deduce  
$$|\partial_r\partial_\mu\sigma|\lesssim \min(\mu r|\log<r>|,\frac{|\log<r>|}{r}).$$
\end{proof}
By taking an extra derivative with respect to $\mu$, we can obtain bounds on $\partial_\mu^2 Q_\mu$ and $\phi_{\partial_\mu^2 Q_\mu}$.
\begin{proposition}\label{proposition:comp3}
The following bounds hold:
\begin{align}\label{phid2muq}
|\phi_{\partial^2_\mu Q_\mu}|\lesssim r^2|\log<r>|,
\end{align}
and
\begin{align}\label{d2qmu}
\partial^2_\mu Q_\mu=\Big[\phi_{\partial^2_\mu Q_\mu}+\Big(\phi_{\partial_\mu Q_\mu}-\frac{r^2}{2}\Big)^2\Big]Q_\mu.
\end{align}
\end{proposition}

\begin{proof}

We differentiate \eqref{phiqdmu} with respect to $\mu$,
\begin{align}\label{eqphid2qmu}
\phi_{\partial^2_\mu Q_\mu}=\partial_\mu^2\sigma.
\end{align}
And by differentiating \eqref{difest} in $\mu$ and using $\partial_\mu Q_\mu=(\phi_{\partial_\mu Q_\mu}-\frac{r^2}{2})Q_\mu$ we get,
\begin{align}
\partial^2_\mu \sigma &=-f_0(r)\int_0^r\Big[\partial^2_\mu \sigma(Q_\mu-Q)+Q_\mu\Big(\partial_\mu \sigma-\frac{r^2}{2}\Big)^2\Big]f_1\tau d\tau\nonumber\\
&+f_1(r)\int_0^r \Big[\partial^2_\mu \sigma(Q_\mu-Q)+Q_\mu\Big(\partial_\mu \sigma-\frac{r^2}{2}\Big)^2\Big]f_0\tau d\tau.
\end{align}
Finally, using that $|f_0|\leq 1$,\quad $|f_1|\lesssim |\log<r>|$,\quad $|Q_\mu-Q|\lesssim  \min(1,\mu r^2)Q$, \\
\quad $|Q_\mu|\lesssim Q$, and $|\partial_\mu \sigma|\lesssim \min\Big(\mu r^2|\log<r>|, |\log(r)|^2\Big)$, we deduce that
\begin{align}
|\partial_\mu^2\sigma|\lesssim r^2|\log<r>| +|\log<r>|\int_0^r|\partial^2_\mu\sigma| \min(\mu\tau^2,1)\tau Q d\tau.
\end{align}

 Hence, 
\eqref{phid2muq} follows from Gronwall inequality.
To justify the extra derivative that we took on $\partial_\mu\sigma$ with respect to $\mu$ one can do the same calculation above on the finite difference $\frac{\partial_\mu\sigma(\mu)-\partial_\mu\sigma(\mu_0)}{\mu-\mu_0}$ and then pass to the limit $\mu\rightarrow\mu_0$ .
 To conclude, \eqref{d2qmu} follows from differentiating \eqref{dmuqmu} with respect to $\mu$.
\end{proof}

\begin{proposition}\label{proposition:comp4}
We have the following bounds:
\begin{align}\label{phid3muq}
|\phi_{\partial^3_\mu Q_\mu}|\lesssim r^4|\log<r>|,
\end{align}
\begin{align}\label{d3qmu}
\partial^3_\mu Q_\mu=\Big[\Big(\phi_{\partial^2_\mu Q_\mu}+\Big(\phi_{\partial_\mu Q_\mu}-\frac{r^2}{2}\Big)^2+2\phi_{\partial^2_\mu Q_\mu}\Big)\Big(\phi_{\partial_\mu Q_\mu}-\frac{r^2}{2}\Big)+\phi_{\partial^3_\mu Q_\mu}\Big]Q_\mu.
\end{align}
\end{proposition}

\begin{proof}
We differentiate \eqref{eqphid2qmu} with respect to $\mu$.
Hence,
$$\phi_{\partial^3_\mu Q_\mu}=\partial_\mu^3\sigma.$$
By processing as in the proof of Proposition \ref{proposition:comp3} one can easily deduce \eqref{phid3muq}, and the details are left to the reader. To conclude, by differentiating 
$$\partial^2_\mu Q_\mu=\Big[\phi_{\partial^2_\mu Q_\mu}+\Big(\phi_{\partial_\mu Q_\mu}-\frac{r^2}{2}\Big)^2\Big]Q_\mu$$ 
with respect to $\mu$, \eqref{d3qmu} follows. 
\end{proof}

In the next following Lemmas we estimate the mass and the second moment of $Q_\mu$ which is fundamental for either the application of the implicit function theorem in Lemma \ref{lemma:mod} or the derivation of the law of $\mu$ in Lemma \ref{lemma:mu_law}.
\begin{lemma}\label{lemma:appM_2mom}
We have the following expansions for the mass of $Q_\mu$:
\begin{align}\label{approxM_Q_mu}
M(\mu)=\int_{\RR^2}Q_\mu dy=8\pi+2\mu\log(\mu)+O(\mu),
\end{align}
\begin{align}\label{Massdmuqmu}
M'(\mu)=\int_{\RR^2}\partial_\mu Q_\mu dy=2\log(\mu)+O(1),
\end{align}
\begin{align}\label{Massd2muqmu}
M''(\mu)=\int_{\RR^2}\partial^2_\mu Q_\mu dy=O(\frac{1}{\mu}).
\end{align}
\end{lemma}

\begin{proof}
To prove \eqref{approxM_Q_mu} we decompose $Q_\mu=Q+Q(e^{\mu\phi_{T_1}+\sigma-\mu\frac{r^2}{2}}-1)$ and plug it in $M=\int_{\RR^2}Q_\mu dy$.
Hence,
\begin{align}
M(\mu)&=\int_{\RR^2}Qdy+\int_{\RR^2}Q(e^{\phi_{Q_\mu}-\mu\frac{|y|^2}{2}}-1)dy\nonumber\\
&=8\pi+\int_0^{\frac{1}{\sqrt{\mu}}}Q(\mu\phi_{T_1}+\sigma-\mu\frac{r^2}{2}+O(\mu^2r^4))rdr+\int_{\frac{1}{\sqrt{\mu}}}^{+\infty}(Q_\mu-Q)rdr,
\end{align}
Since $|Q_\mu-Q|\lesssim Q$ for $r\geq \frac{1}{\sqrt{\mu}}$ we get that
$$\int_{\frac{1}{\sqrt{\mu}}}^{+\infty}|Q_\mu-Q|rdr\lesssim\int_{\frac{1}{\sqrt{\mu}}}^{+\infty}Q rdr\lesssim \mu.$$
Hence, by using $|\sigma|\lesssim \mu^2r^2|\log<r>|$ for $r\leq \frac{1}{\sqrt{\mu}}$ and $|\phi_{T_1}|\lesssim |\log<r>|^2$ we deduce that
\begin{align}\label{mestqmu}
M(\mu)&=8\pi-\frac{\mu}{2}\int_0^{\frac{1}{\sqrt{\mu}}}Qr^3dr+O(\mu)=8\pi+2\mu\log(\mu)+O(\mu).
\end{align}
To prove \eqref{Massdmuqmu}, we use that $\partial_\mu Q_\mu=(\phi_{\partial_\mu Q_\mu}-\frac{r^2}{2})Q_\mu$.
It follows by using Propositions \ref{proposition:approxQmu} and \ref{proposition:comp2} that
\begin{align*}
\int_{\RR^2}\partial_\mu Q_\mu dy&=\int_0^{\frac{1}{\sqrt{\mu}}}(\phi_{\partial_\mu Q_\mu}-\frac{r^2}{2})Q(1+\mu\phi_{T_1}+\sigma-\mu\frac{r^2}{2}+O(\mu^2r^4))rdr\\
&+\int_{\frac{1}{\sqrt{\mu}}}^\infty(\phi_{\partial_\mu Q_\mu}-\frac{r^2}{2})Q_\mu rdr.
\end{align*}
Since, $|\sigma|\lesssim \min(\mu^2r^2|\log<r>|,\mu|\log<r>|^2)$, $|\phi_{T_1}|\lesssim \min\Big(r^4|\log<r>|,|\log<r>|^2\Big)$ and $|\phi_{\partial_\mu Q_\mu}|\lesssim r^2$, it follows that
\begin{align}\label{r3qmu}
\int_{\frac{1}{\sqrt{\mu}}}^\infty\Big|\phi_{\partial_\mu Q_\mu}-\frac{r^2}{2}\Big|Q_\mu rdr\lesssim \int_{\frac{1}{\sqrt{\mu}}}^\infty r^3 Q_\mu dr\lesssim \int_{\frac{1}{\sqrt{\mu}}}^\infty r^3Qe^{\mu(|\log<r>|^2-\frac{r^2}{2})}dr\lesssim 1,
\end{align}
and
$$\int_0^{\frac{1}{\sqrt{\mu}}}\Big|\phi_{\partial_\mu Q_\mu}-\frac{r^2}{2}\Big|Q\Big|\mu\phi_{T_1}+\sigma-\mu\frac{r^2}{2}+O(\mu^2r^4)\Big|rdr+\int_0^{\frac{1}{\sqrt{\mu}}}\Big|\phi_{\partial_\mu Q_\mu}\Big|Qrdr\lesssim 1.$$
Hence,
\begin{align}
M'(\mu)=\int_{\RR^2}\partial_\mu Q_\mu dy&=-\int_0^{\frac{1}{\sqrt{\mu}}}Q\frac{r^3}{2}dr+O(1)=2\log(\mu)+O(1).
\end{align}
To prove \eqref{Massd2muqmu} we use Proposition \ref{proposition:comp3} to deduce that,
\begin{align}
M''(\mu)=\int_{\RR^2}\partial^2_\mu Q_\mu dy=\int_{\RR^2}\Big[\phi_{\partial_\mu^2 Q_\mu}+\Big(\phi_{\partial_\mu Q_\mu}-\frac{r^2}{2}\Big)^2\Big]Q_\mu dy.
\end{align}
From Proposition \ref{proposition:comp4} we have for all $r\geq0$:
$$r^4-1\lesssim \phi_{\partial_\mu^2 Q_\mu}+\Big(\phi_{\partial_\mu Q_\mu}-\frac{r^2}{2}\Big)^2\lesssim r^4+1.$$
Hence,
\begin{align}
\int_0^\infty r^5 Q_\mu dr-1\lesssim M''(\mu)\lesssim\int_0^\infty r^5Q_\mu(r)dr+1.
\end{align}
Since, $Q_\mu=Qe^{\mu\phi_{T_1}+\sigma-\mu\frac{r^2}{2}}$ we obtain that
\begin{align}\label{qmuy4}
\int_0^\infty r^5 Q_\mu dr=\int_0^\infty r^5 Qe^{\mu\phi_{T_1}+\sigma-\mu\frac{r^2}{2}}dr=\frac{1}{\mu}\int_0^\infty \tau^5 \frac{1}{(\mu+\tau^2)^2}e^{\mu\phi_{T_1}(\frac{\tau}{\sqrt{\mu}})+\sigma(\frac{\tau}{\sqrt{\mu}})-\frac{\tau^2}{2}}d\tau,
\end{align} 
and since $\int_0^\infty \tau^5 \frac{1}{(\mu+\tau^2)^2}e^{\mu\phi_{T_1}(\frac{\tau}{\sqrt{\mu}})+\sigma(\frac{\tau}{\sqrt{\mu}})-\frac{\tau^2}{2}}d\tau$ is uniformly bounded  with respect to $\mu$ we deduce that \eqref{Massd2muqmu} holds,
which concludes the proof of Lemma \ref{lemma:appM_2mom}.
\end{proof}

\begin{lemma}\label{lemma:approx2momqmu}
We have the following expression on the second moment of $Q_\mu$ and the following bounds on the derivatives of the second moment of $Q_\mu$ with respect to $\mu$:
\begin{align}\label{2M_Q_mu}
\int_{\RR^2}Q_\mu(y)|y|^2 dy=\frac{2M(\mu)}{\mu}\Big(1-\frac{M(\mu)}{8\pi}\Big).
\end{align}
There exists $C>0$ such that
\begin{align}\label{2momdmu}
\frac{-C}{\mu}\leq \int_{\RR^2}\partial_\mu Q_\mu(y)|y|^2 dy\leq \frac{-1}{C\mu},
\end{align}
\begin{align}\label{2momd2mu}
\frac{1}{C\mu^2}\leq \int_{\RR^2}\partial^2_\mu Q_\mu(y)|y|^2 dy\leq \frac{C}{\mu^2},
\end{align}
\begin{align}\label{2momd3mu}
\frac{-C}{\mu^3}\leq \int_{\RR^2}\partial^3_\mu Q_\mu(y)|y|^2 dy\leq \frac{-1}{C\mu^3}.
\end{align}
\end{lemma}

\begin{proof}
\begin{itemize}
\item Proof of \eqref{2M_Q_mu}: \newline
We start by calculating the second moment of $Q_\mu$.We multiply \eqref{eq_Q_mu} by $|y|^2$ and integrate
\begin{align}
\int_{\RR^2}\Delta Q_\mu |y|^2dy-\int_{\RR^2}\nabla\cdot(Q_\mu\nabla\phi_{Q_\mu})|y|^2dy=-\mu\int_{\RR^2}\nabla\cdot(yQ_\mu)|y|^2dy,
\end{align}
then by integration by parts and using that $\nabla\phi_{Q_\mu}=\frac{1}{|\cdot|}\star Q_\mu$ we get
\begin{align}
4M+\frac{1}{2\pi}\int_{\RR^2\times\RR^2}\frac{2y\cdot(y-x)}{|x-y|^2}Q_\mu(x)Q_\mu(y)dxdy&=2\mu\int_{\RR^2}Q_\mu(y)|y|^2 dy\\
4M+\frac{1}{2\pi}\int_{\RR^2\times\RR^2}\frac{(x-y)\cdot(y-x)}{|x-y|^2}Q_\mu(x)Q_\mu(y)dxdy&=2\mu\int_{\RR^2}Q_\mu(y)|y|^2 dy\\
4M-\frac{M^2}{2\pi}&=2\mu\int_{\RR^2}Q_\mu(y)|y|^2 dy,
\end{align}
which implies \eqref{2M_Q_mu}.
We have
\item Proof of \eqref{2momdmu}. \newline
\begin{align}
\int_{\RR^2}\partial_\mu Q_\mu(y)|y|^2 dy=\int_{\RR^2}\Big(\phi_{\partial_\mu Q_\mu}-\frac{r^2}{2}\Big) Q_\mu(y)|y|^2 dy.
\end{align}
Since, $Q_\mu=Qe^{\mu\phi_{T_1}+\sigma-\mu\frac{r^2}{2}}$, and $\phi_{\partial_\mu Q_\mu}=\phi_{T_1}+\partial_\mu\sigma$, it follows
\begin{align}
\int_{\RR^2}\partial_\mu Q_\mu(y)|y|^2 dy=\int_0^\infty\Big(\phi_{T_1}+\partial_\mu\sigma-\frac{r^2}{2}\Big) Qe^{\mu\phi_{T_1}+\sigma-\mu\frac{r^2}{2}}r^3dr.
\end{align}
By using that $|\phi_{T_1}|\lesssim  \min\Big( r^4|\log<r>|,|\log<r>|^2\Big)$ and\\ $|\partial_\mu\sigma|\lesssim \min\Big(\mu r^2|\log<r>|,|\log<r>|^2\Big)$, we deduce that for all $r\geq0$ 
$$-r^2\lesssim\phi_{T_1}+\partial_\mu\sigma-\frac{r^2}{2}\lesssim-r^2.$$
Hence,
\begin{align}
&-\int_0^\infty r^5 Qe^{\mu\phi_{T_1}+\sigma-\mu\frac{r^2}{2}}dr\lesssim\int_{\RR^2}\partial_\mu Q_\mu(y)|y|^2 dy\lesssim-\int_0^\infty r^5 Qe^{\mu\phi_{T_1}+\sigma-\mu\frac{r^2}{2}}dr.
\end{align}
Notice that 
$$\int_0^\infty r^5 Qe^{\mu\phi_{T_1}+\sigma-\mu\frac{r^2}{2}}dr=\frac{1}{\mu}\int_0^\infty \tau^5 \frac{1}{(\mu+\tau^2)^2}e^{\mu\phi_{T_1}(\frac{\tau}{\sqrt{\mu}})+\sigma(\frac{\tau}{\sqrt{\mu}})-\frac{\tau^2}{2}}d\tau,$$ 
and since $\int_0^\infty \tau^5 \frac{1}{(\mu+\tau^2)^2}e^{\mu\phi_{T_1}(\frac{\tau}{\sqrt{\mu}})+\sigma(\frac{\tau}{\sqrt{\mu}})-\frac{\tau^2}{2}}d\tau$ is uniformly bounded  with respect to $\mu$ we deduce \eqref{2momdmu}.
\item Proof of \eqref{2momd2mu}.\newline 
Notice first that
$$\partial_\mu^2Q_\mu=\Big[\phi_{\partial^2_\mu Q_\mu}+\Big(\phi_{\partial_\mu Q_\mu}-\frac{r^2}{2}\Big)^2\Big]Q_\mu.$$
Hence, by using $|\phi_{T_1}|\lesssim  \min\Big( r^4|\log<r>|,|\log<r>|^2\Big)$, $$|\partial_\mu\sigma|\lesssim \min\Big(\mu r^2|\log<r>|,|\log<r>|^2\Big) \mbox{ and } |\phi_{\partial^2_\mu Q_\mu}|\lesssim r^2|\log<r>|,$$ we deduce 
for all $r\geq0$
$$r^4-1\lesssim\phi_{\partial^2_\mu Q_\mu}+\Big(\phi_{\partial_\mu Q_\mu}-\frac{r^2}{2}\Big)^2\lesssim r^4+1.$$
\begin{align*}
&\int_0^\infty (r^7-r^3) Qe^{\mu\phi_{T_1}+\sigma-\mu\frac{r^2}{2}}dr\lesssim\int_{\RR^2}\partial_\mu Q_\mu(y)|y|^2 dy\\
&\lesssim\int_0^\infty (r^7+r^3) Qe^{\mu\phi_{T_1}+\sigma-\mu\frac{r^2}{2}}dr.
\end{align*}
From \eqref{mestqmu} and \eqref{r3qmu} we get
$$ \int_0^\infty r^3 Qe^{\mu\phi_{T_1}+\sigma-\mu\frac{r^2}{2}}dr=O(|\log \mu|).$$
Notice also that
$$\int_0^\infty r^7 Qe^{\mu\phi_{T_1}+\sigma-\mu\frac{r^2}{2}}dr=\frac{1}{\mu^2}\int_0^\infty \tau^7 \frac{1}{(\mu+\tau^2)^2}e^{\mu\phi_{T_1}(\frac{\tau}{\sqrt{\mu}})+\sigma(\frac{\tau}{\sqrt{\mu}})-\frac{\tau^2}{2}}d\tau,$$ 
and since $\int_0^\infty \tau^7 \frac{1}{(\mu+\tau^2)^2}e^{\mu\phi_{T_1}(\frac{\tau}{\sqrt{\mu}})+\sigma(\frac{\tau}{\sqrt{\mu}})-\frac{\tau^2}{2}}d\tau$ is uniformly bounded  with respect to $\mu$ we deduce \eqref{2momd2mu}.
The last step is to prove \eqref{2momd3mu}.
\item Proof of \eqref{2momd3mu}.\newline
If we use that $$\partial_\mu^3Q_\mu=\Big[\Big(\phi_{\partial^2_\mu Q_\mu}+\Big(\phi_{\partial_\mu Q_\mu}-\frac{r^2}{2}\Big)^2+2\phi_{\partial^2_\mu Q_\mu}\Big)\Big(\phi_{\partial_\mu Q_\mu}-\frac{r^2}{2}\Big)+\phi_{\partial^3_\mu Q_\mu}\Big]Q_\mu,$$ then by proceeding as in the proof of the previous inequalities one can easily deduce \eqref{2momd3mu} and the rest of the proof is left to the reader.
\end{itemize}
\end{proof}
Now we are ready to decide what the good approximate profile will be. Indeed, we choose the following profile: 
\begin{equation}\label{profile}
\tilde{Q}_\mu=Q_\mu- \tilde{\mu}\partial_\mu Q_\mu,
\end{equation}
where 
\begin{equation}\label{mutilde}
\tilde{\mu}=\frac{\int_{\RR^2}Q_\mu-8\pi}{\int_{\RR^2}\partial_\mu Q_\mu}=\mu+O(\frac{\mu}{|\log\mu|}),
\end{equation}
which can be deduced easily from Lemma \ref{lemma:appM_2mom}.
This choice can be justified by 2 reasons:
The first one is that the mass of $\varepsilon$ will be zero and this implies that 
$$(\mathcal{M}_\mu^y\varepsilon,\partial_M n_\infty)=(\varepsilon,1)=0.$$
The second reason is that if we plug $\tilde{Q_\mu}$ in
\begin{align*}
\Delta\tilde{Q}_\mu+\mu\Lambda\tilde{Q}_\mu-\nabla\cdot(\tilde{Q}_\mu\nabla\phi_{\tilde{Q}_\mu})&=-\tilde{\mu}{\mathcal{L}_\mu^y}(\partial_\mu Q_\mu)-\tilde{\mu}^2\nabla\cdot(\partial_\mu Q_\mu\nabla\phi_{\partial_\mu Q_\mu})\\
&=\tilde{\mu}\Lambda Q_\mu -\tilde{\mu}^2\nabla\cdot(\partial_\mu Q_\mu\nabla\phi_{\partial_\mu Q_\mu}),
\end{align*}
the identity $-{\mathcal{L}_\mu^y}(\partial_\mu Q_\mu)=\Lambda Q_\mu$ yields that the error $E=\tilde{\mu}\Lambda Q_\mu -\tilde{\mu}^2\nabla\cdot(\partial_\mu Q_\mu\nabla\phi_{\partial_\mu Q_\mu})$ is of order $\mu^2$ when it is projected on ${\mathcal{M}_\mu^y}\varepsilon$. 
 Indeed, since $(\varepsilon,|\cdot|^2)=(\varepsilon,1)=0$ it follows that $$(\mathcal{M}_\mu^y(\Lambda Q_\mu),\varepsilon)=(2-\frac{M}{2\pi}-\mu|\cdot|^2,\varepsilon)=0,$$
which implies 
$$(E,{\mathcal{M}_\mu^y}\varepsilon)_{L^2}=-\tilde{\mu}^2(\nabla\cdot(\partial_\mu Q_\mu\nabla\phi_{\partial_\mu Q_\mu}),{\mathcal{M}_\mu^y}\varepsilon)_{L^2}.$$
This cancellation makes the projection of our error on $\varepsilon$ with the inner product $(\cdot,\cdot)_{{\mathcal{M}_\mu^y}}$ of order $\mu^2$ which is fundamental for closing the energy estimate.
\section{Derivation of the law of $\mu(s(t))$}
Recall that,
$$ v(y,s) = \tilde{Q}_\mu(y)+\varepsilon(y).$$
The conservation of second moment in the $x$ coordinate transaltes into:
\begin{align}\label{cons2moment}
\int_{\RR^2}v|y|^2 dy=\frac{I}{\mu R(t)^2},
\end{align}
where $I$ is the second moment of the initial data $u_0$. 
We will make an abuse of notation and use $\mu(t)$ for $\mu(s(t))$.
\begin{lemma}\label{lemma:mu_law}
Let $v$ being the solution of \eqref{v2rescal} with $8\pi$ mass and $\varepsilon$ satisfying $$(\varepsilon,|\cdot|^2)_{L^2}=0.$$
Then, 
 \begin{equation}\label{murelwithepsilon}
\frac{I}{\mu R(t)^2}=-\frac{M}{2\pi}\log(\mu)+O(1)
\end{equation}
and
 \begin{equation}\label{mu_lawt}
\mu(t)=\frac{2\pi I}{M(2t+1)\log(2t+1)+O(t\log\log(2t+1))}.
\end{equation}
If we consider $\mu$ as a function of $s$, we have:
\begin{equation}\label{mu_laws}
\Big|\mu(s)-\frac{1}{2s}\Big|\leq\frac{C'}{s\log(s)}.
\end{equation}
\end{lemma}

\begin{proof}[Proof of Lemma \ref{lemma:mu_law}]

If we combine \eqref{approxM_Q_mu} and \eqref{2M_Q_mu} then
\begin{align}\label{approx2M_Q_mu}
\int_{\RR^2}Q_\mu(y)|y|^2dy=\frac{2M}{\mu}\Big(\frac{-2\mu\log(\mu)+O(\mu)}{8\pi}\Big)=-\frac{M\log(\mu)}{2\pi}+O(1).
\end{align}
It follows,from \eqref{cons2moment}, $(\varepsilon,|\cdot|^2)=0$, and \eqref{approx2M_Q_mu} that
\begin{align}\label{mulawproof}
\frac{I}{R^2\mu}&=\int_{\RR^2}\tilde{Q}_\mu(y)|y|^2dy=-\frac{M}{2\pi}\log(\mu)+O(1)-\tilde{\mu}\int_{\RR^2}\partial_\mu Q_\mu |y|^2dy.
\end{align}
From \eqref{2momdmu} we deduce that
\begin{align}\label{2momdmuQmu}
\int_{\RR^2}\partial_\mu Q_\mu(y)|y|^2dy=O\Big(\frac{1}{\mu}\Big).
\end{align}
Hence, by using $\tilde{\mu}=\mu+O\Big(\frac{\mu}{|\log\mu|}\Big)$ it follows that
\begin{align}\label{mulaw}
\frac{I}{R^2\mu}&=\int_{\RR^2}\tilde{Q}_\mu(y)|y|^2dy=-\frac{M}{2\pi}\log(\mu)+O(1),
\end{align}
and taking the $\log$ of \eqref{mulaw}, we get that,
$$\log\Big(\frac{1}{\mu}\Big)-\log(2t+1)=\log\log\Big(\frac{1}{\mu}\Big)+O(1).$$
Consequently,
$$\log\log\Big(\frac{1}{\mu}\Big)=\log\log(2t+1)+O(1),$$
which imply \eqref{mu_lawt}.
To find $\mu(s)$ we use that 

\begin{align}
\frac{ds}{dt}=\frac{1}{\mu R^2}=\frac{M\log(2t+1)+O(\log\log(2t+1))}{2\pi I}.
\end{align}
Hence, if we integrate we find $$s(t)=\frac{2(2t+1)M\log(2t+1)+O(t\log\log(2t+1))}{2\pi I},$$ and \eqref{mu_laws} follows
with $C'$ in \eqref{mu_laws} a uniform constant with respect to $A_1$, and hence if we choose $A_1$ sufficiently large we get $C'<\frac{A_1}{100}$.
\end{proof}
\bigskip

Now we need to find a bound on $\mu_s$ for the energy estimates.
\begin{lemma}\label{museq}
\begin{align}\label{musbound}
\mu_s= -2\mu^2+O(\frac{\mu^2}{|\log\mu|}).
\end{align}
\end{lemma}
\begin{proof}
To prove \eqref{musbound}, we differentiate $\frac{I}{R^2\mu}=\int_{\RR^2}\tilde{Q}_\mu(y)|y|^2dy$ with respect to $s$.
Hence,
\begin{align}
I\frac{d}{ds}\Big(\frac{1}{R^2\mu}\Big)=\mu_s\int_{\RR^2}\partial_\mu\tilde{Q}_\mu(y)|y|^2dy.
\end{align}
To derive \eqref{musbound} we need to estimate $\int_{\RR^2}\partial_\mu\tilde{Q}_\mu(y)|y|^2dy$.
Indeed, since $\partial_\mu\tilde{Q}_\mu(y)=\partial_\mu Q_\mu-\partial_\mu \tilde{\mu}\partial_\mu Q_\mu-\tilde{\mu}\partial^2_\mu Q_\mu$.
And from Lemma \ref{lemma:appM_2mom} we compute 
\begin{align}\label{dmumutilde}
\partial_\mu\tilde{\mu}=1+O\Big(\frac{1}{|\log\mu|}\Big).
\end{align}
Hence,
\begin{align}\label{dmutildeqmu}
\partial_\mu \tilde{Q}_\mu=-\mu\partial_\mu^2Q_\mu+O\Big(\frac{\mu\partial_\mu^2 Q_\mu+\partial_\mu Q_\mu}{|\log\mu|}\Big).
\end{align}
From Lemma \ref{lemma:approx2momqmu} we deduce
\begin{align}\label{2momdmutildeqmu}
\int_{\RR^2}\partial_\mu\tilde{Q}_\mu|y|^2dy&=-\mu\int_{\RR^2}\partial_\mu^2 Q_\mu|y|^2 dy+O\Big(\int_{\RR^2}\frac{\mu\partial_\mu^2 Q_\mu+\partial_\mu Q_\mu}{|\log\mu|}|y|^2dy\Big)\nonumber\\
&=-\mu\int_{\RR^2}\partial_\mu^2 Q_\mu|y|^2 dy+O\Big(\frac{1}{\mu|\log\mu|}\Big)=-\frac{C}{\mu}+O\Big(\frac{1}{\mu|\log\mu|}\Big),
\end{align}
and similarly we obtain from Lemma \ref{lemma:approx2momqmu},
\begin{align}\label{2momd2mutildeqmu}
\int_{\RR^2}\partial_\mu^2\tilde{Q}_\mu|y|^2dy=\frac{C}{\mu^2}+O\Big(\frac{1}{\mu^2|\log\mu|}\Big),
\end{align}
where $C>0$ is a constant depending on the constant $C$ in \eqref{2momd2mu} and \eqref{2momd3mu}.
Finally, since $\frac{dt}{ds}=\mu R^2$ and $R'(t)R(t)=1$ it follows
\begin{align}\label{exact2momdmutildeqmu}
-I\Big(\frac{\frac{d}{ds}(\mu R(t)^2)}{\mu^2R^4}\Big)&=-I\Big(\frac{\mu_s}{\mu^2R^2}+2\frac{R'(t)R}{R^2}\Big)=-I\Big(\frac{\mu_s}{\mu^2R^2}+\frac{2}{R^2}\Big)\nonumber\\
&=\mu_s\int_{\RR^2}\partial_\mu\tilde{Q}_\mu(y)|y|^2dy=\mu_s O\Big(\frac{1}{\mu}\Big),
\end{align}
which concludes the proof.
\end{proof}

\section{Bounds on the potential $\nabla\phi_{\varepsilon}$}

 With the orthogonality conditions  $$(\varepsilon,1)=(\varepsilon,|\cdot|^2)=0,$$ $$(\mathcal{M}_\mu^y\varepsilon,\varepsilon)=\|\varepsilon\|_{L^2_{Q_\mu}}^2-\|\nabla\phi_{\varepsilon}\|_{L^2}^2$$ does not control $\|\varepsilon\|_{L^2_{Q_\mu}}$ uniformly in $\mu$. It turns out that we have a remarkable nonlinear structure that yields a control of $\|\nabla\phi_{\varepsilon}\|_{L^2}^2$ and hence $\|\varepsilon\|_{L^2_{Q_\mu}}$. 
\begin{proposition}\label{proposition:boundonphieps}
\begin{align}
\int_{\RR^2}|\nabla\phi_{\varepsilon}|^2dy=O(\mu).
\end{align}
\end{proposition}
\begin{proof}
We take the $L^2$ inner product of (\ref{eqn:linQmu}) with $|\cdot|^2$.
\begin{align*}
\frac{d}{ds}(\varepsilon,|\cdot|^2) &=(\mathcal{L}_{\mu}\varepsilon,|\cdot|^2)-(\nabla\cdot(\varepsilon\nabla\phi_{\varepsilon}),|\cdot|^2)+\frac{{\mu}_s}{2{\mu}}(\Lambda\varepsilon,|\cdot|^2)+(\Theta_{\mu}(\varepsilon),|\cdot|^2)\\
&+(F,|\cdot|^2)+(E,|\cdot|^2).
\end{align*}
 Since $$(\varepsilon,|\cdot|^2)=(\varepsilon,1)=0,$$ 
\begin{align*}
(\mathcal{L}_{\mu}\varepsilon,|\cdot|^2)&=-2(\tilde{Q}_\mu\nabla{\mathcal{M}_\mu^y}\varepsilon,y)=2({\mathcal{M}_\mu^y}\varepsilon,\Lambda \tilde{Q}_{\mu})\\
&=2(\varepsilon,{\mathcal{M}_\mu^y}\Lambda \tilde{Q}_{\mu})=(2-\frac{M}{2\pi})(\varepsilon,1)-\mu(\varepsilon,|y|^2)=0,
\end{align*}
 and
 $$( \Lambda \varepsilon,|\cdot|^2)=-2(\varepsilon,|\cdot|^2)=0.$$
In addition,
\begin{align}
-(\nabla\cdot(\varepsilon\nabla\phi_{\varepsilon}),|\cdot|^2)=2\int_{\RR^2}\varepsilon y\cdot\nabla\phi_{\varepsilon}dy=\int_{\RR^2}|\nabla\phi_{\varepsilon}|^2dy.
\end{align}
The last equality comes from the fact that $\int_{\RR^2}\varepsilon dy=0$, $-\Delta\phi_{\varepsilon}=\varepsilon$.
Since $F=\frac{\mu_s}{2\mu}\Lambda \tilde{Q}_\mu -\mu_s\partial_\mu \tilde{Q}_\mu$, to estimate $(F,|\cdot|^2)$ we first calculate 
$(\Lambda \tilde{Q}_{\mu},|\cdot|^2)$ and $(\partial_{\mu} \tilde{Q}_{\mu},|\cdot|^2)$.
Indeed, from \eqref{mulawproof} we get
\begin{align}
(\Lambda \tilde{Q}_{\mu},|\cdot|^2)=-2(\tilde{Q}_{\mu},|\cdot|^2)=-\frac{2I}{R^2\mu},
\end{align}
and by using \eqref{2momdmutildeqmu} we deduce
\begin{align}
\mu_s(\partial_{\mu} \tilde{Q}_{\mu},|\cdot|^2)=-I\Big(\frac{\mu_s}{\mu^2R^2}+\frac{2}{R^2}\Big).
\end{align}
Hence,
$$(F,|\cdot|^2)=(,|\cdot|^2)=\frac{2I}{R^2}=-\frac{M\mu\log\mu}{\pi}+O(\mu).$$
Furthermore, since $E=\tilde{\mu}\Lambda Q_{\mu}-\tilde{\mu}^2\nabla\cdot(\partial_{\mu}Q_{\mu}\nabla\phi_{\partial_{\mu} Q_{\mu}})$ and $\tilde{\mu}=\mu+O\Big(\frac{\mu}{|\log\mu|}\Big)$, it follows that 
\begin{align}
(E,|\cdot|^2)&=-2\tilde{\mu}(Q_\mu,|\cdot|^2)+2\tilde{\mu}^2(\nabla\phi_{\partial_{\mu} Q_{\mu}}\cdot y,\partial_\mu Q_\mu)\nonumber\\
&=-2\mu(Q_\mu,|\cdot|^2)+2\tilde{\mu}^2(\nabla\phi_{\partial_{\mu} Q_{\mu}}\cdot y,\partial_\mu Q_\mu)+O(\mu)\nonumber\\
&=\frac{M\mu\log\mu}{\pi}+2\tilde{\mu}^2(\nabla\phi_{\partial_{\mu} Q_{\mu}}\cdot y,\partial_\mu Q_\mu)+O(\mu).
\end{align}
From Proposition \ref{proposition:comp2} we get that
$$2\tilde{\mu}^2(\nabla\phi_{\partial_{\mu} Q_{\mu}}\cdot y,\partial_\mu Q_\mu)\lesssim \mu^2|\log\mu|^2.$$
Hence,
\begin{align}
(F+E,|\cdot|^2)=\underbrace{-\frac{M\mu\log\mu}{\pi}+\frac{M\mu\log\mu}{\pi}}_{\mbox{ Cancellation }}+O(\mu).
\end{align}
To estimate $(\Theta_\mu(\varepsilon),|\cdot|^2)$ we use that $\tilde{Q}_{\mu}-Q_{\mu}=-\tilde{\mu}\partial_\mu Q_\mu$.
Hence,
\begin{align}
(\Theta_{\mu}(\varepsilon),|\cdot|^2)=-\tilde{\mu}\Big[(\partial_\mu Q_\mu,\nabla\phi_{\varepsilon}\cdot y)+(\varepsilon,\nabla_{\phi_{\partial_\mu Q_\mu}}\cdot y)\Big].
\end{align}
It follows that,

\begin{align}
|(\Theta_{\mu}(\varepsilon),|\cdot|^2)|\lesssim \mu\Big[\|\nabla\phi_{\varepsilon}\|_{L^2}\Big(\int_{\RR^2}|\partial_\mu Q_\mu|^2|y|^2dy\Big)^{\frac{1}{2}}+\|\varepsilon\|_{L^2_{Q_\mu}}\Big(\int_{\RR^2}\frac{|\nabla\phi_{\partial_\mu Q_\mu}|^2|y|^2}{Q_\mu}dy\Big)^{\frac{1}{2}}\Big].
\end{align}
Thanks to $|\nabla\phi_{\partial_\mu Q_\mu}|\lesssim\min(\mu |y||\log |y||,\frac{|\log |y||}{|y|})$,\quad $\partial_\mu Q_\mu\lesssim Q_\mu|y|^2$ and the bootstrap assumption we deduce
\begin{align}
|(\Theta_{\mu}(\varepsilon),|\cdot|^2)|\lesssim\sqrt{A}\mu^{\frac{3}{2}}\sqrt{|\log \mu|}.
\end{align}
Hence, if we assume that $\mu$ was choosen small enough, we deduce that
\begin{align}\label{musapprox}
\int_{\RR^2}|\nabla\phi_{\varepsilon}|^2dy\leq C\mu,
\end{align}
with $C$ uniform in $\mu$ which concludes the proof.
\end{proof}

\section{Energy estimates}
We want to prove the following energy bound.
\begin{proposition}\label{proposition:energybound} 
\begin{align}
\frac{1}{2}\frac{d}{ds}({\mathcal{M}_\mu^y}\varepsilon,\varepsilon)+\mu(K_2-\frac{1}{2}-\delta)({\mathcal{M}_\mu^y}\varepsilon,\varepsilon)\leq C(\delta,A)\mu^2,
\end{align}
with $C(A,\delta)\lesssim \frac{1}{\delta}(\sqrt{A}(1+\delta)+1)$, where $\delta>0$ is a sufficiently small constant and $A$ is the bootstrap constant.
\end{proposition}

\begin{proof}
We multiply (\ref{eqn:linQmu})  by ${\mathcal{M}_\mu^y}\varepsilon$ and integrate
\begin{align}\label{estimMe}
\frac{1}{2}\frac{d}{ds}({\mathcal{M}_\mu^y}\varepsilon,\varepsilon)&=-\int_{\RR^2}Q_{\mu}\bigl|\nabla{\mathcal{M}_\mu^y}\varepsilon\bigr|^2dy+\underbrace{\frac{\mu_s}{2\mu}(\Lambda\varepsilon,{\mathcal{M}_\mu^y}\varepsilon)-\frac{\mu_s}{2}(\frac{\partial_{\mu} Q_{\mu}\varepsilon}{Q_{\mu}^2},\varepsilon)}_{I}\nonumber\\
&+\underbrace{(E,{\mathcal{M}_\mu^y}\varepsilon)}_{II}
+\underbrace{(\Theta_{\mu},{\mathcal{M}_\mu^y}\varepsilon)}_{III}+\underbrace{(N(\varepsilon),{\mathcal{M}_\mu^y}\varepsilon)}_{IV}+\underbrace{(F,{\mathcal{M}_\mu^y}\varepsilon)}_{V}.\nonumber\\
\end{align}

Now we start to estimate each term of \eqref{estimMe}.

\subsection{Estimation of $I$}
\subsubsection{Estimation of $(\frac{\partial_{\mu} Q_{\mu}\varepsilon}{Q_{\mu}^2},\varepsilon)$}

We will use the notation
$$<y>=\sqrt{1+|y|^2}.$$
Notice first that by Proposition \ref{proposition:approxQmu}  
$$\frac{\partial_\mu Q_\mu}{Q_\mu}=\phi_{\partial_\mu Q_\mu}-\frac{r^2}{2}=-\frac{r^2}{2}+\min(O((\log(r))^2),O(\mu r^2|\log<r>|)).$$
This term will be very helpful since it has the good sign if we can control
$$\int_{\RR^2}|\log(<y>)|^2\frac{\varepsilon^2}{Q_{\mu}}dy.$$
Indeed, we can prove that 
$$\int_{\RR^2}|\log(<y>)|^2\frac{\varepsilon^2}{Q_{\mu}}dy\lesssim A\sqrt{\mu}+\frac{\sqrt{A}}{\sqrt{\mu}}\Big(\int_{\RR^2}Q_\mu|\nabla \mathcal{M}_\mu^y\varepsilon|^2dy\Big)^{\frac{1}{2}}.$$
We first use that $\frac{\varepsilon}{Q_\mu}=\mathcal{M}_\mu^y\varepsilon+\phi_{\varepsilon}$, which implies
\begin{align}
\int_{\RR^2}|\log(<y>)|^2\frac{\varepsilon^2}{Q_{\mu}}dy=\int_{\RR^2}|\log(<y>)|^2\varepsilon\mathcal{M}_\mu^y\varepsilon dy+\int_{\RR^2}|\log(<y>)|^2\varepsilon \phi_{\varepsilon}dy.
\end{align}
Then by  Proposition \ref{proposition:ineqpot} and bootstrap assumption
$$\int_{\RR^2}|\log(<y>)|^2\varepsilon \phi_{\varepsilon}dy\leq \|\phi_{\varepsilon}\|_{L^\infty}\|\varepsilon\|_{L^2_{Q_\mu}}\||\log(<y>)|^2\sqrt{Q_\mu}\|_{L^2}\lesssim A\mu.$$
To bound the other term we use bootstrap assumption, Lemma \ref{lemma:Mprop} and Lemma \ref{lemma:PCineq},
\begin{align}
\int_{\RR^2}|\log(<y>)|^2\varepsilon\mathcal{M}_\mu^y\varepsilon dy&\leq\|\varepsilon\|_{L^2_{Q_\mu}}\Big(\int_{\RR^2}|\log(<y>)|^4Q_\mu|\mathcal{M}_\mu^y\varepsilon|^2dy\Big)^{\frac{1}{2}} \nonumber\\
&\lesssim \|\varepsilon\|_{L^2_{Q_\mu}}\Big(\int_{\RR^2}|y|^2 Q_\mu|\mathcal{M}_\mu^y\varepsilon|^2dy\Big)^{\frac{1}{2}}\nonumber\\
&\lesssim \|\varepsilon\|_{L^2_{Q_\mu}}\Big(\frac{4}{\mu^2}\int_{\RR^2}Q_\mu|\nabla \mathcal{M}_\mu^y\varepsilon|^2dy+\frac{12}{\mu}\int_{\RR^2}Q_\mu|\mathcal{M}_\mu^y\varepsilon|^2dy\Big)^{\frac{1}{2}}\nonumber\\
&\lesssim \frac{\sqrt{A}}{\sqrt{\mu}}\Big(\int_{\RR^2}Q_\mu|\nabla \mathcal{M}_\mu^y\varepsilon|^2dy\Big)^{\frac{1}{2}}+A\sqrt{\mu},
\end{align}
which yields 
$$\int_{\RR^2}|\log(<y>)|^2\frac{\varepsilon^2}{Q_{\mu}}dy\lesssim A\sqrt{\mu}+\frac{\sqrt{A}}{\sqrt{\mu}}\Big(\int_{\RR^2}Q_\mu|\nabla \mathcal{M}_\mu^y\varepsilon|^2 dy\Big)^{\frac{1}{2}}.$$
Hence, if we combine the previous inequality with $\mu_s=-2\mu^2+O(\frac{\mu^2}{|\log\mu|})$
\begin{align}\label{time_der_enrgy}
-\frac{{\mu}_s}{2}(\frac{\partial_{\mu} Q_{\mu}\varepsilon}{Q_{\mu}^2},\varepsilon)=-\frac{\mu^2}{2}\int_{\RR^2}|y|^2\frac{\varepsilon^2}{Q_{\mu}}dy+O(\mu^{\frac{5}{2}}).
\end{align}

\subsubsection{Estimation of $(\Lambda\varepsilon,{\mathcal{M}_\mu^y}\varepsilon)$}

\begin{align}
(\Lambda\varepsilon,{\mathcal{M}_\mu^y}\varepsilon) &= 2({\mathcal{M}_\mu^y}\varepsilon,\varepsilon)+\int_{\RR^2}  y\cdot\nabla \varepsilon {\mathcal{M}_\mu^y}\varepsilon dy \nonumber\\
&= 2({\mathcal{M}_\mu^y}\varepsilon,\varepsilon)+\int_{\RR^2} y\cdot\nabla\varepsilon\frac{\varepsilon}{Q_{\mu}}dy -\int_{\RR^2}y\cdot\nabla\varepsilon\phi_{\varepsilon}dy.
\end{align}

We calculate each term separately
\begin{eqnarray}
\int_{\RR^2} y\cdot\nabla\varepsilon\frac{\varepsilon}{Q_{\mu}}dy &=&\frac{1}{2}\int_{\RR^2}\frac{y\cdot\nabla(\varepsilon^2)}{Q_{\mu}} dy= -\int_{\RR^2}\frac{\varepsilon^2}{Q_{\mu}} dy+\frac{1}{2}\int_{\RR^2}\varepsilon^2\frac{y\cdot\nabla Q_{\mu}}{Q_{\mu}^2} dy.
\end{eqnarray}
In addition,
\begin{eqnarray*}
-\int_{\RR^2}y\cdot\nabla\varepsilon\phi_{\varepsilon}dy&=&2\int_{\RR^2}\varepsilon\phi_{\varepsilon}+ \int_{\RR^2}\varepsilon y\cdot\nabla\phi_{\varepsilon}dy.
\end{eqnarray*}
Since for $\int_{\RR^2}\varepsilon dy=0$ we have that 
$$\int_{\RR^2}\varepsilon\phi_{\varepsilon}=\int_{\RR^2}|\nabla\phi_{\varepsilon}|^2dy,$$
by using $-\Delta\phi_{\varepsilon}=\varepsilon$ and integration by parts we deduce that
\begin{eqnarray*}
-\int_{\RR^2}y\cdot\nabla\varepsilon\phi_{\varepsilon}dy=2\int_{\RR^2}\varepsilon\phi_{\varepsilon}dy+\frac{1}{2}\int_{\RR^2}|\nabla\phi_{\varepsilon}|^2dy=\frac{5}{2}\int_{\RR^2}|\nabla\phi_{\varepsilon}|^2dy.
\end{eqnarray*}
Hence, 
\begin{align}\label{Lamdaeps}
(\Lambda\varepsilon,{\mathcal{M}_\mu^y}\varepsilon)&=2({\mathcal{M}_\mu^y}\varepsilon,\varepsilon)-\int_{\RR^2}\frac{\varepsilon^2}{Q_{\mu}} dy+\frac{1}{2}\int_{\RR^2}\varepsilon^2\frac{y\cdot\nabla Q_{\mu}}{Q_{\mu}^2} dy+\frac{5}{2}\int_{\RR^2}|\nabla\phi_{\varepsilon}|^2dy\nonumber\\ 
&=\int_{\RR^2}\frac{\varepsilon^2}{Q_{\mu}} dy+\frac{1}{2}\int_{\RR^2}\varepsilon^2\frac{y\cdot\nabla Q_{\mu}}{Q_{\mu}^2} dy+\frac{1}{2}
\int_{\RR^2}|\nabla\phi_{\varepsilon}|^2dy.
\end{align}
Since $\frac{y\cdot\nabla Q_{\mu}}{Q_{\mu}}=y\cdot\nabla\phi_{Q_\mu}-\mu|y|^2$ and $y\cdot\nabla\phi_{Q_\mu}>y\cdot\nabla\phi_{Q}>-4$ we get that
\begin{align}
(\Lambda\varepsilon,{\mathcal{M}_\mu^y}\varepsilon)&=\int_{\RR^2}\frac{\varepsilon^2}{Q_{\mu}} dy+\frac{1}{2}
\int_{\RR^2}|\nabla\phi_{\varepsilon}|^2dy-\frac{\mu}{2}\int_{\RR^2}\frac{|y|^2\varepsilon^2}{Q_{\mu}} dy+\frac{1}{2}\int_{\RR^2}\frac{\varepsilon^2y\cdot\nabla\phi_{Q_\mu}}{Q_{\mu}} dy\nonumber\\
&\geq\int_{\RR^2}\frac{\varepsilon^2}{Q_{\mu}} dy+\frac{1}{2}
\int_{\RR^2}|\nabla\phi_{\varepsilon}|^2dy-\frac{\mu}{2}\int_{\RR^2}\frac{|y|^2\varepsilon^2}{Q_{\mu}} dy-2\int_{\RR^2}\frac{\varepsilon^2}{Q_{\mu}} dy\nonumber\\
&\geq -\int_{\RR^2}\frac{\varepsilon^2}{Q_{\mu}} dy+\frac{1}{2}
\int_{\RR^2}|\nabla\phi_{\varepsilon}|^2dy-\frac{\mu}{2}\int_{\RR^2}\frac{|y|^2\varepsilon^2}{Q_{\mu}} dy
\end{align}
By using $\mu_s=-2\mu^2+O\Big(\frac{\mu^2}{|\log \mu|}\Big)$ we obtain that
\begin{align}
\frac{\mu_s}{2\mu}(\Lambda\varepsilon,{\mathcal{M}_\mu^y}\varepsilon)&\leq\mu\int_{\RR^2}\frac{\varepsilon^2}{Q_{\mu}} dy-\frac{\mu}{2}\int_{\RR^2}|\nabla\phi_{\varepsilon}|^2dy+\frac{\mu^2}{2}\int_{\RR^2}\frac{|y|^2\varepsilon^2}{Q_{\mu}} dy\nonumber\\
&\leq \mu({\mathcal{M}_\mu^y} \varepsilon,\varepsilon)+\frac{\mu}{2}\int_{\RR^2}|\nabla\phi_{\varepsilon}|^2dy+\frac{\mu^2}{2}\int_{\RR^2}\frac{|y|^2\varepsilon^2}{Q_{\mu}} dy.
\end{align}
Since $\int_{\RR^2}|\nabla\phi_{\varepsilon}|^2dy\leq C\mu$ it holds
\begin{align}
\frac{\mu_s}{2\mu}(\Lambda\varepsilon,{\mathcal{M}_\mu^y}\varepsilon)&\leq \mu({\mathcal{M}_\mu^y} \varepsilon,\varepsilon)+\frac{\mu^2}{2}\int_{\RR^2}\frac{|y|^2\varepsilon^2}{Q_{\mu}} dy+\frac{C\mu^2}{2},
\end{align}
with $\delta>0$ a sufficiently small constant independent of $\mu$.
The previous inequality and \eqref{time_der_enrgy} imply the following bound on $I$,
\begin{align*}
I&=\frac{\mu_s}{2\mu}(\Lambda\varepsilon,{\mathcal{M}_\mu^y}\varepsilon)-\frac{\mu_s}{2}(\frac{\partial_{\mu} Q_{\mu}\varepsilon}{Q_{\mu}^2},\varepsilon)\leq \mu({\mathcal{M}_\mu^y} \varepsilon,\varepsilon)+\frac{\mu^2}{2}\int_{\RR^2}\frac{|y|^2\varepsilon^2}{Q_{\mu}} dy-\frac{\mu^2}{2}\int_{\RR^2}\frac{|y|^2\varepsilon^2}{Q_{\mu}} dy\\
&+\frac{C\mu^2}{2}+O(\mu^{\frac{5}{2}})
\nonumber\\
&\leq\mu({\mathcal{M}_\mu^y} \varepsilon,\varepsilon)+\frac{\mu^2}{\delta},
\end{align*}
with $\delta>0$ sufficiently small.

\subsection{ Estimation of $II$}

\begin{remark}
First we notice that 
$E=\tilde{\mu}\Lambda Q_{\mu}-\tilde{\mu}^2\nabla\cdot(\partial_{\mu}Q_{\mu}\nabla\phi_{\partial_{\mu} Q_{\mu}})$,
and
$$(\Lambda Q_\mu,{\mathcal{M}_\mu^y}\varepsilon)=({\mathcal{M}_\mu^y}(\Lambda Q_\mu),\varepsilon)=(2-\frac{M}{2\pi})(1,\varepsilon)-\mu(|\cdot|^2,\varepsilon)=0.$$
Thanks to this cancellation, we are able to close our energy estimates.
\end{remark}
Hence, by using$|\nabla\phi_{\partial_\mu Q_\mu}|\lesssim \min\Big(\mu|y||\log|y||,\frac{|\log|y||}{|y|}\Big)$ and $|\partial_\mu Q_\mu|\lesssim |y|^2Q_\mu$ ( these inequalities are from Proposition \ref{proposition:comp2})
\begin{align*}
(E,{\mathcal{M}_\mu^y}\varepsilon)&=-\tilde{\mu}^2(\nabla\cdot(\partial_{\mu}Q_{\mu}\nabla\phi_{\partial_{\mu} Q_{\mu}}),{\mathcal{M}_\mu^y}\varepsilon)\\
&\lesssim\mu^2\Big(\int_{\RR^2}\frac{|\partial_\mu Q_\mu|^2|\nabla\phi_{\partial_\mu Q_\mu}|^2}{Q_\mu}dy\Big)^{\frac{1}{2}}\Big(\int_{\RR^2}Q_\mu|\nabla{\mathcal{M}_\mu^y}\varepsilon|^2dy\Big)^{\frac{1}{2}}\nonumber\\
&\lesssim \mu^2|\log\mu|^{\frac{3}{2}}\Big(\int_{\RR^2}Q_\mu|\nabla{\mathcal{M}_\mu^y}\varepsilon|^2dy\Big)^{\frac{1}{2}}\lesssim \mu\int_{\RR^2}Q_\mu|\nabla{\mathcal{M}_\mu^y}\varepsilon|^2dy+\mu^3|\log\mu|^3.
\end{align*}

\subsection{ Estimation of $III$}
Since,
$$\Theta_\mu(\varepsilon)=-\tilde{\mu}\nabla\cdot[\partial_{\mu}Q_{\mu}\nabla\phi_\varepsilon+\varepsilon\nabla\phi_{\partial_{\mu}Q_{\mu}}],$$
and by $|\nabla\phi_{\partial_\mu Q_\mu}|\lesssim \min\Big(\mu|y||\log|y||,\frac{|\log|y||}{|y|}\Big)$, $|\partial_\mu Q_\mu|\lesssim |y|^2Q_\mu$, \begin{align*}
(\Theta_\mu(\varepsilon),{\mathcal{M}_\mu^y}\varepsilon)&=\tilde{\mu}\Big[\int_{\RR^2}\partial_\mu Q_\mu \nabla\phi_{\varepsilon}\cdot\nabla{\mathcal{M}_\mu^y}\varepsilon dy+\int_{\RR^2}\varepsilon\nabla\phi_{\partial_\mu Q_\mu}\cdot\nabla{\mathcal{M}_\mu^y}\varepsilon dy\Big]\nonumber\\
&\lesssim \mu\Big(\int_{\RR^2}Q_\mu|\nabla{\mathcal{M}_\mu^y}\varepsilon|^2dy\Big)^{\frac{1}{2}}\\
&\times\Big[\Big(\int_{\RR^2}\frac{|\partial_\mu Q_\mu|^2|\nabla\phi_{\varepsilon}|^2}{Q_\mu}dy\Big)^{\frac{1}{2}}+\Big(\int_{\RR^2}\frac{|\nabla\phi_{\partial_\mu Q_\mu}|^2|\varepsilon|^2}{Q_\mu}dy\Big)^{\frac{1}{2}}\Big]\nonumber\\
&\lesssim \mu [\|\nabla\phi_{\varepsilon}\|_{L^2}+\|\varepsilon\|_{L^2_{Q_\mu}}]\Big(\int_{\RR^2}Q_\mu|\nabla{\mathcal{M}_\mu^y}\varepsilon|^2dy\Big)^{\frac{1}{2}}.
\end{align*}
Hence, by the bootstrap assumption $\|\varepsilon\|_{L^2_{Q_\mu}}\leq\sqrt{A\mu}$ and $\|\nabla\phi_{\varepsilon}\|_{L^2}=O(\sqrt{\mu})$ we get
$$ III=(\Theta_\mu(\varepsilon),{\mathcal{M}_\mu^y}\varepsilon)\lesssim \mu^{\frac{5}{2}}+\sqrt{\mu}\int_{\RR^2}Q_\mu|\nabla{\mathcal{M}_\mu^y}\varepsilon|^2dy.$$

\subsection{ Estimation of $IV$}

For the nonlinear term we will use a sobolev inequality not on $\varepsilon$ but on $\mathcal{M}_\mu^y\varepsilon$ because we control  only the gradient of $\mathcal{M}_\mu^y\varepsilon$. To do so we use the identity $\varepsilon=Q_\mu\mathcal{M}_\mu^y\varepsilon+\phi_{\varepsilon}Q_\mu$.
\begin{align*}
(N(\varepsilon),{\mathcal{M}_\mu^y}\varepsilon) &= -\int_{\RR^2}\varepsilon\nabla\phi_\varepsilon\cdot\nabla{\mathcal{M}_\mu^y}\varepsilon dy=-\int_{\RR^2}Q_\mu\mathcal{M}_\mu^y\varepsilon\nabla\phi_\varepsilon\cdot\nabla{\mathcal{M}_\mu^y}\varepsilon dy\\
&-\int_{\RR^2}Q_\mu\phi_{\varepsilon}\nabla\phi_\varepsilon\cdot\nabla{\mathcal{M}_\mu^y}\varepsilon dy.
\end{align*}
To estimate $\int_{\RR^2}Q_\mu\phi_{\varepsilon}\nabla\phi_\varepsilon\cdot\nabla{\mathcal{M}_\mu^y}\varepsilon dy$ we use Proposition \ref{proposition:ineqpot} and $\sup_{y\in\RR^2}Q_\mu=8$,
\begin{align}
\int_{\RR^2}Q_\mu\phi_{\varepsilon}\nabla\phi_\varepsilon\cdot\nabla{\mathcal{M}_\mu^y}\varepsilon dy&\leq \|\phi_{\varepsilon}\|_{L^\infty}\|\sqrt{Q_\mu}\nabla\phi_{\varepsilon}\|_{L^2}\Big(\int_{\RR^2}Q_\mu|\nabla{\mathcal{M}_\mu^y}\varepsilon|^2dy\Big)^{\frac{1}{2}}.
\end{align}
By Morrey inequality, $-\Delta\phi_{\varepsilon}=\varepsilon$, and interpolation we get 
\begin{align}
\|\phi_{\varepsilon}\|_{L^\infty}&\lesssim \|\phi_{\varepsilon}\|_{L^4}+\|\nabla\phi_{\varepsilon}\|_{L^4}\lesssim \|\nabla\phi_{\varepsilon}\|_{L^2}^{\frac{1}{2}}(\|\Delta\phi_{\varepsilon}\|_{L^2}^{\frac{1}{2}}+\|\phi_{\varepsilon}\|_{L^2}^{\frac{1}{2}})\nonumber\\
&\lesssim \|\nabla\phi_{\varepsilon}\|_{L^2}^{\frac{1}{2}}(\|\varepsilon\|_{L^2_{Q_\mu}}^{\frac{1}{2}}+\|\phi_{\varepsilon}\|_{L^2}^{\frac{1}{2}}).
\end{align}
To control $\|\phi_{\varepsilon}\|_{L^2}$ we use Lemma \ref{lemma:poissonbound}
\begin{align}\label{phiL2}
\|\phi_{\varepsilon}\|_{L^2}^{\frac{1}{2}}\lesssim\|\varepsilon\|_{L^2_{Q_\mu}}^{\frac{1}{2}}.
\end{align}
Hence, by using the bootstrap assumption $\|\varepsilon\|_{L^2_{Q_\mu}}\leq\sqrt{A\mu}$
\begin{align}
\int_{\RR^2}Q_\mu\phi_{\varepsilon}\nabla\phi_\varepsilon\cdot\nabla{\mathcal{M}_\mu^y}\varepsilon dy&\lesssim \|\nabla\phi_{\varepsilon}\|_{L^2}^{\frac{3}{2}}\|\varepsilon\|_{L^2_{Q_\mu}}^{\frac{1}{2}}\Big(\int_{\RR^2}Q_\mu|\nabla{\mathcal{M}_\mu^y}\varepsilon|^2dy\Big)^{\frac{1}{2}}\nonumber\\
&\lesssim A^{\frac{1}{4}}\mu\Big(\int_{\RR^2}Q_\mu|\nabla{\mathcal{M}_\mu^y}\varepsilon|^2dy\Big)^{\frac{1}{2}}\nonumber\\
&\lesssim \frac{\sqrt{A}}{\delta}\mu^2+\delta\int_{\RR^2}Q_\mu|\nabla{\mathcal{M}_\mu^y}\varepsilon|^2dy,
\end{align}
with $\delta$ a sufficiently small constant.
Since the size of $\int_{\RR^2}Q_\mu\mathcal{M}_\mu^y\varepsilon\nabla\phi_\varepsilon\cdot\nabla{\mathcal{M}_\mu^y}\varepsilon dy$ is critical
we will take out the biggest component of $\varepsilon$ which is in the direction of $\Lambda Q_\mu$.
To do so, we introduce the following decomposition of $\varepsilon$: 
$$\varepsilon=\alpha_\mu\Lambda Q_\mu+\hat{\varepsilon},$$
and we fix $\alpha_\mu$ such that 
\begin{align}\label{orthcond}
 \int_{\RR^2}\phi_{\Lambda Q_\mu}Q_\mu\mathcal{M}_\mu^y\hat{\varepsilon}dy=0.
 \end{align}
This specific choice of the orthogonality condition is crucial here.
Indeed, this choice will ensure that $\alpha_\mu$ has a weak dependence in $A$ (the bootstrap constant) and that the constant $C$ in \eqref{Hardyineq} is uniform with respect to $\mu$.
$$\int_{B}Q_\mu|\mathcal{M}_\mu^y\varepsilon|^2 dy\leq C\int_{\RR^2}Q_\mu|\nabla{\mathcal{M}_\mu^y}\varepsilon|^2dy.$$
First let us prove that
$$|\alpha_\mu|\lesssim \sqrt{\mu}.$$
Recall from \eqref{Mlambdaqmu} that 
$$\phi_{\Lambda Q_\mu}Q_\mu=\Lambda Q_\mu-[(2-\frac{M}{2\pi})-\mu|\cdot|^2]Q_\mu.$$  
On the one hand, using Proposition \ref{proposition:algebraic_identities} and Lemma \ref{lemma:appM_2mom} we have
 \begin{align}
 \int_{\RR^2}\phi_{\Lambda Q_\mu}Q_\mu\mathcal{M}_\mu^y\varepsilon dy&=\alpha_\mu  \int_{\RR^2}\Big[\Lambda Q_\mu-[(2-\frac{M}{2\pi})-\mu|\cdot|^2]Q_\mu\Big]\mathcal{M}_\mu^y\Lambda Q_\mu dy\nonumber\\
 &= \alpha_\mu\int_{\RR^2}\Big[\Lambda Q_\mu-\Big[\Big(2-\frac{M}{2\pi}\Big)-\mu|\cdot|^2\Big]Q_\mu\Big]\Big(2-\frac{M}{2\pi}-\mu|y|^2\Big)dy.
 \end{align}
Since, $(\Lambda Q_\mu,1)=0$ and $(\Lambda Q_\mu,|\cdot|^2)=-2(Q_\mu,|\cdot|^2)=-\frac{M\log\mu}{2\pi}+O(1)$ it follows
 \begin{align}
 int_{\RR^2}\phi_{\Lambda Q_\mu}Q_\mu\mathcal{M}_\mu^y\varepsilon dy&=\alpha_\mu\Big(\Big(2-\frac{M}{2\pi}\Big)^2M+O(\mu|\log\mu|)\Big).
 \end{align}
On the other hand, using that $\int_{\RR^2}\varepsilon dy=\int_{\RR^2}\varepsilon |y|^2 dy=(\mathcal{M}_\mu^y\varepsilon,\Lambda Q_\mu)_{L^2}=0$, we get
\begin{align*}
\int_{\RR^2}\phi_{\Lambda Q_\mu}Q_\mu\mathcal{M}_\mu^y\varepsilon dy&= \int_{\RR^2}\Big[\Lambda Q_\mu-\Big[\Big(2-\frac{M}{2\pi}\Big)-\mu|\cdot|^2\Big]Q_\mu\Big]\mathcal{M}_\mu^y\varepsilon dy\nonumber\\
&=-\int_{\RR^2}\Big[\Big(2-\frac{M}{2\pi}\Big)-\mu|\cdot|^2\Big]Q_\mu \mathcal{M}_\mu^y\varepsilon  dy\\
&=-\int_{\RR^2}\Big[\Big(2-\frac{M}{2\pi}\Big)-\mu|\cdot|^2\Big]Q_\mu\phi_{\varepsilon}.
 \end{align*}
Hence, using that $\varepsilon$ has average zero and the decay of $Q_\mu$, we get that:
\begin{align}\label{alphamu}
|\alpha_\mu|\lesssim \int_{\RR^2}Q_\mu|\phi_\varepsilon |dy\lesssim \|\nabla\phi_{\varepsilon}\|_{L^2}.
\end{align}
Moerever, we prove below in the following proposition that the difference between $({\mathcal{M}_\mu^y}\varepsilon,\varepsilon)$ and $({\mathcal{M}_\mu^y}\hat{\varepsilon},\hat{\varepsilon})$ is of order $\mu^2|\log\mu|$.
\begin{proposition}\label{proposition:difbetE_hatE}
If $(\varepsilon,1)_{L^2}=(\varepsilon,|\cdot|^2)_{L^2}=0$, $\varepsilon=\alpha_\mu\Lambda Q_\mu+\hat{\varepsilon}$ and  $$\int_{\RR^2}\phi_{\Lambda Q_\mu}Q_\mu\mathcal{M}_\mu^y\hat{\varepsilon}dy=0$$ then
\begin{align}
(\mathcal{L}_\mu^y\varepsilon,\mathcal{M}_\mu^y\varepsilon)=(\mathcal{L}_\mu^y\hat{\varepsilon},\mathcal{M}_\mu^y\hat{\varepsilon})+O(\mu^3|\log\mu|),
\end{align}
\begin{align}
({\mathcal{M}_\mu^y}\varepsilon,\varepsilon)=({\mathcal{M}_\mu^y}\hat{\varepsilon},\hat{\varepsilon})+O(\mu^2|\log\mu|).
\end{align}
\end{proposition}
The proof is in the appendix.
We plug the decomposition of $\varepsilon=\alpha_\mu\Lambda Q_\mu+\hat{\varepsilon}$ in 
$$\int_{\RR^2}Q_\mu\mathcal{M}_\mu^y\varepsilon\nabla\phi_\varepsilon\cdot\nabla{\mathcal{M}_\mu^y}\varepsilon dy.$$
It follows:
\begin{align*}
\int_{\RR^2}Q_\mu\mathcal{M}_\mu^y\varepsilon\nabla\phi_\varepsilon\cdot\nabla{\mathcal{M}_\mu^y}\varepsilon dy&=\alpha_\mu\int_{\RR^2}Q_\mu\mathcal{M}_\mu^y(\Lambda Q_\mu)\nabla\phi_\varepsilon\cdot\nabla{\mathcal{M}_\mu^y}\hat{\varepsilon} dy\\
&+\alpha_\mu\int_{\RR^2}Q_\mu\mathcal{M}_\mu^y\hat{\varepsilon}\nabla\phi_\varepsilon\cdot \nabla\mathcal{M}_\mu^y(\Lambda Q_\mu)dy\nonumber\\
&+\int_{\RR^2}Q_\mu\mathcal{M}_\mu^y\hat{\varepsilon}\nabla\phi_\varepsilon\cdot\nabla{\mathcal{M}_\mu^y}\hat{\varepsilon} dy\\
&+\alpha_\mu^2\int_{\RR^2}Q_\mu\mathcal{M}_\mu^y(\Lambda Q_\mu)\nabla\phi_\varepsilon\cdot \nabla\mathcal{M}_\mu^y(\Lambda Q_\mu) dy
\end{align*}
Hence, by using $\mathcal{M}_\mu^y(\Lambda Q_\mu)=2-\frac{M}{2\pi}-\mu|y|^2$, we deduce 
\begin{align*}
&\int_{\RR^2}Q_\mu\mathcal{M}_\mu^y\varepsilon\nabla\phi_\varepsilon\cdot\nabla{\mathcal{M}_\mu^y}\varepsilon dy=\underbrace{\alpha_\mu\int_{\RR^2}Q_\mu\mathcal{M}_\mu^y(\Lambda Q_\mu)\nabla\phi_\varepsilon\cdot\nabla{\mathcal{M}_\mu^y}\hat{\varepsilon} dy}_{{IV}_1}\\
&-\underbrace{2\mu\alpha_\mu\int_{\RR^2}Q_\mu\mathcal{M}_\mu^y\hat{\varepsilon}\nabla\phi_\varepsilon\cdot ydy}_{{IV}_2}
+\underbrace{\int_{\RR^2}Q_\mu\mathcal{M}_\mu^y\hat{\varepsilon}\nabla\phi_\varepsilon\cdot\nabla{\mathcal{M}_\mu^y}\hat{\varepsilon} dy}_{{IV}_3}\\
&-\underbrace{2\mu\alpha_\mu^2\int_{\RR^2}Q_\mu\mathcal{M}_\mu^y(\Lambda Q_\mu)\nabla\phi_\varepsilon\cdot y dy}_{{IV}_4}.
\end{align*}
We first start to estimate the term ${IV}_1$.

\subsubsection{ Estimation of ${IV}_1$}

Since, $\mathcal{M}_\mu^y(\Lambda Q_\mu)=2-\frac{M}{2\pi}-\mu|y|^2$ and by \eqref{alphamu} we obtain
\begin{align*}
{IV}_1&\lesssim|\alpha_\mu|\int_{\RR^2}Q_\mu \nabla\phi_\varepsilon\cdot\nabla{\mathcal{M}_\mu^y}\hat{\varepsilon} dy+|\alpha_\mu|\mu\int_{\RR^2}Q_\mu|y|^2 \nabla\phi_\varepsilon\cdot\nabla{\mathcal{M}_\mu^y}\hat{\varepsilon} dy\nonumber\\
&\lesssim \sqrt{\mu}\|\nabla\phi_\varepsilon\|_{L^2} \Big(\int_{\RR^2}Q_\mu|\nabla{\mathcal{M}_\mu^y}\hat{\varepsilon}|^2dy\Big)^{\frac{1}{2}}\\
&+\mu^{\frac{3}{2}}\|\sqrt{Q_\mu}|y|^2\nabla\phi_\varepsilon\|_{L^2} \Big(\int_{\RR^2}Q_\mu|\nabla{\mathcal{M}_\mu^y}\hat{\varepsilon}|^2dy\Big)^{\frac{1}{2}}.
\end{align*}
Notice that $\sqrt{Q_\mu}|y|^2$ is uniformly bounded and $\|\nabla\phi_\varepsilon\|_{L^2}\lesssim \sqrt{\mu}$, which imply
\begin{align}
{IV}_1&\lesssim \delta\int_{\RR^2}Q_\mu|\nabla{\mathcal{M}_\mu^y}\hat{\varepsilon}|^2dy + \frac{1}{\delta}\mu^2+\frac{1}{\delta}\mu^4
\lesssim \delta\int_{\RR^2}Q_\mu|\nabla{\mathcal{M}_\mu^y}\hat{\varepsilon}|^2dy + \frac{1}{\delta}\mu^2.
\end{align}
And if we use Proposition \ref{proposition:difbetE_hatE} we obtain
\begin{align}
{IV}_1\lesssim \delta\int_{\RR^2}Q_\mu|\nabla{\mathcal{M}_\mu^y}\varepsilon|^2dy + \frac{1}{\delta}\mu^2 + \mu^3|\log\mu|.
\end{align}

\subsubsection{ Estimation of ${IV}_2$}

Now we estimate ${IV}_2$.
We remark that $\sqrt{Q_\mu}|y|$ is uniformly bounded, which implies
\begin{align*}
{IV}_2&\lesssim \mu|\alpha_\mu|\Big(\int_{\RR^2}Q_\mu|{\mathcal{M}_\mu^y}\hat{\varepsilon}|^2dy\Big)^{\frac{1}{2}}\|\sqrt{Q_\mu}|y|\nabla\phi_\varepsilon\|_{L^2}\\
&\lesssim \mu|\alpha_\mu|\Big(\int_{\RR^2}Q_\mu|{\mathcal{M}_\mu^y}\hat{\varepsilon}|^2dy\Big)^{\frac{1}{2}}\|\nabla\phi_\varepsilon\|_{L^2}.
\end{align*}
Notice that 
\begin{align}
\int_{\RR^2}Q_\mu|{\mathcal{M}_\mu^y}\hat{\varepsilon}|^2dy&\lesssim \int_{\RR^2}Q_\mu|{\mathcal{M}_\mu^y}\varepsilon|^2dy+\alpha_\mu^2\int_{\RR^2}Q_\mu|{\mathcal{M}_\mu^y}\Lambda Q_\mu|^2dy\nonumber\\
&\lesssim \int_{\RR^2}Q_\mu|{\mathcal{M}_\mu^y}\varepsilon|^2dy+C\alpha_\mu^2,
\end{align}
which implies
\begin{align}\label{Mmuhatepsilon}
\int_{\RR^2}Q_\mu|{\mathcal{M}_\mu^y}\hat{\varepsilon}|^2dy&\lesssim \int_{\RR^2}Q_\mu|{\mathcal{M}_\mu^y}\varepsilon|^2dy+\mu\lesssim A\mu.
\end{align}
Hence, if we combine all the previous inequalities it follows
\begin{align}
{IV}_2&\lesssim \mu^2A^{\frac{1}{2}}\|\nabla\phi_\varepsilon\|_{L^2}\lesssim\mu^{\frac{5}{2}}A^{\frac{1}{2}}.
\end{align}

\subsubsection{Estimation of ${IV}_4$}
To estimate ${IV}_4$, we use $\mathcal{M}_\mu^y(\Lambda Q_\mu)=2-\frac{M}{2\pi}-\mu|y|^2$:

\begin{align}
{IV}_4&\lesssim \mu\alpha_\mu^2\Big[\int_{\RR^2}Q_\mu|\nabla\phi_{\varepsilon}||y| dy+\mu\int_{\RR^2}Q_\mu|\nabla\phi_{\varepsilon}||y|^3 dy\Big]\nonumber\\
&\lesssim\mu\alpha_\mu^2\|\nabla\phi_{\varepsilon}\|_{L^2}\bigl[\|Q_\mu|y|\|_{L^2}+\|Q_\mu|y|^3\|_{L^2}\bigr]\lesssim\mu^{\frac{5}{2}}.
\end{align}

\subsubsection{ Estimation of ${IV}_3$ }

To control the term ${IV}_3$ we first use Gargliano-Niremberg inequality
$$\|f\|_{L^{4}}\lesssim \|f\|_{L^2}^{\frac{1}{2}}\|\nabla f\|_{L^2}^{\frac{1}{2}},$$
with $f=\sqrt{Q_\mu}\mathcal{M}_\mu^y\hat{\varepsilon}$.
Hence,
\begin{align}
\int_{\RR^2}Q_\mu\mathcal{M}_\mu^y\hat{\varepsilon}\nabla\phi_\varepsilon\cdot\nabla{\mathcal{M}_\mu^y}\hat{\varepsilon} dy&\leq \Big(\int_{\RR^2}Q_\mu|\nabla{\mathcal{M}_\mu^y}\hat{\varepsilon}|^2dy\Big)^{\frac{1}{2}}\Big(\int_{\RR^2}Q_\mu|{\mathcal{M}_\mu^y}\hat{\varepsilon}|^2|\nabla\phi_{\varepsilon}|^2dy\Big)^{\frac{1}{2}}\nonumber\\
&\leq \Big(\int_{\RR^2}Q_\mu|\nabla{\mathcal{M}_\mu^y}\hat{\varepsilon}|^2dy\Big)^{\frac{1}{2}}\|\sqrt{Q_\mu}{\mathcal{M}_\mu^y}\hat{\varepsilon}\|_{L^4}\|\nabla\phi_{\varepsilon}\|_{L^4}\nonumber\\
&\leq \Big(\int_{\RR^2}Q_\mu|\nabla{\mathcal{M}_\mu^y}\hat{\varepsilon}|^2dy\Big)^{\frac{1}{2}}\|\nabla\phi_{\varepsilon}\|_{L^4}\\
&\times\|\sqrt{Q_\mu}{\mathcal{M}_\mu^y}\hat{\varepsilon}\|_{L^2}^{\frac{1}{2}}\|\nabla(\sqrt{Q_\mu}{\mathcal{M}_\mu^y}\hat{\varepsilon})\|_{L^2}^{\frac{1}{2}}.
\end{align}
We also apply Gargliano-Niremberg inequality for $\|\nabla\phi_{\varepsilon}\|_{L^4}$ and use $$-\Delta\phi_{\varepsilon}=\varepsilon, \mbox{ and } L^2_{Q_\mu}\hookrightarrow L^2.$$
\begin{align}\label{estC}
\int_{\RR^2}Q_\mu\mathcal{M}_\mu^y\hat{\varepsilon}\nabla\phi_\varepsilon\cdot\nabla{\mathcal{M}_\mu^y}\hat{\varepsilon} dy&\leq \Big(\int_{\RR^2}Q_\mu|\nabla{\mathcal{M}_\mu^y}\hat{\varepsilon}|^2dy\Big)^{\frac{1}{2}}\|\nabla\phi_{\varepsilon}\|_{L^2}^{\frac{1}{2}}\|\varepsilon\|_{L^2_{Q_\mu}}^{\frac{1}{2
}}\nonumber\\
&\times\|\sqrt{Q_\mu}{\mathcal{M}_\mu^y}\varepsilon\|_{L^2}^{\frac{1}{2}}\|\nabla(\sqrt{Q_\mu}{\mathcal{M}_\mu^y}\hat{\varepsilon})\|_{L^2}^{\frac{1}{2}}.
\end{align}
Now we estimate $\|\nabla(\sqrt{Q_\mu}{\mathcal{M}_\mu^y}\hat{\varepsilon})\|_{L^2}^2$. Since $\frac{|\nabla Q_\mu|^2}{Q_\mu}\lesssim \Big(\mu^2|y|^2+\frac{|y|^2}{(1+|y|^2)^2}\Big)Q_\mu$, we deduce
\begin{align*}
\|\nabla(\sqrt{Q_\mu}{\mathcal{M}_\mu^y}\hat{\varepsilon})\|_{L^2}^2&\lesssim\int_{\RR^2}Q_\mu|\nabla\mathcal{M}_\mu^y\hat{\varepsilon}|^2dy+\int_{\RR^2}\frac{|\nabla Q_\mu|^2}{Q_\mu}|\mathcal{M}_\mu^y\hat{\varepsilon}|^2dy\nonumber\\
&\lesssim\int_{\RR^2}Q_\mu|\nabla\mathcal{M}_\mu^y\hat{\varepsilon}|^2dy+\mu^2\int_{\RR^2}|y|^2Q_\mu|\mathcal{M}_\mu^y\hat{\varepsilon}|^2dy\\
&+\int_{\RR^2}\frac{|y|^2}{(1+|y|^2)^2}Q_\mu |\mathcal{M}_\mu^y\hat{\varepsilon}|^2dy.
\end{align*}
Then to control $\mu^2\int_{\RR^2}|y|^2Q_\mu|\mathcal{M}_\mu^y\hat{\varepsilon}|^2dy$ and $\int_{\RR^2}\frac{|y|^2}{(1+|y|^2)^2}Q_\mu |\mathcal{M}_\mu^y\hat{\varepsilon}|^2dy$, we use Proposition \ref{Proposition:Hardyineq} and Lemma \ref{lemma:PCineq} from the appendix.
Hence,
\begin{align}\label{estC}
\int_{\RR^2}Q_\mu\mathcal{M}_\mu^y\hat{\varepsilon}\nabla\phi_\varepsilon\cdot\nabla{\mathcal{M}_\mu^y}\hat{\varepsilon} dy&\leq \Big(\int_{\RR^2}Q_\mu|\nabla{\mathcal{M}_\mu^y}\hat{\varepsilon}|^2dy\Big)^{\frac{1}{2}}A^{\frac{1}{2}}\mu^{\frac{3}{4}}\Big[\Big(\int_{\RR^2}Q_\mu|\nabla{\mathcal{M}_\mu^y}\hat{\varepsilon}|^2dy\Big)^{\frac{1}{4}}\nonumber\\
&+\mu^{\frac{1}{4}}\Big(\int_{\RR^2}Q_\mu|{\mathcal{M}_\mu^y}\hat{\varepsilon}|^2dy\Big)^{\frac{1}{4}}\Big].
\end{align}
From \eqref{Mmuhatepsilon} and Holder we get
\begin{align}
\int_{\RR^2}Q_\mu\mathcal{M}_\mu^y\hat{\varepsilon}\nabla\phi_\varepsilon\cdot\nabla{\mathcal{M}_\mu^y}\hat{\varepsilon} dy
&\lesssim \delta \int_{\RR^2}Q_\mu|\nabla{\mathcal{M}_\mu^y}\varepsilon|^2dy+\frac{\mu^3 A^2}{\delta}+\mu^3|\log\mu|.
\end{align}
Finally,
$$(N(\varepsilon),{\mathcal{M}_\mu^y}\varepsilon)\lesssim \frac{1+A^{\frac{1}{2}}}{\delta}\mu^2+\delta\int_{\RR^2}Q_\mu|\nabla{\mathcal{M}_\mu^y}\varepsilon|^2dy,$$
with $\delta>0$ a sufficiently small constant.

\subsection{Estimation of $V$}
We notice first 
$$F=\frac{\mu_s}{2\mu}[\Lambda Q_\mu-\tilde{\mu}\Lambda\partial_\mu Q_\mu]-\mu_s[\partial_\mu Q_\mu -\partial_\mu \tilde{\mu} \partial_\mu Q_\mu-\tilde{\mu}\partial_\mu^2 Q_\mu],$$
and since $(\partial_\mu Q_\mu,{\mathcal{M}_\mu^y}\varepsilon)=(\Lambda Q_\mu,{\mathcal{M}_\mu^y}\varepsilon)=0$,
it follows
\begin{align*}
(F,{\mathcal{M}_\mu^y}\varepsilon)&=-\tilde{\mu}\frac{\mu_s}{2\mu}(\Lambda \partial_\mu Q_\mu,{\mathcal{M}_\mu^y}\varepsilon)+\mu_s\tilde{\mu}(\partial_\mu^2 Q_\mu,{\mathcal{M}_\mu^y}\varepsilon)=-\tilde{\mu}\frac{\mu_s}{2\mu}(y\cdot\nabla\partial_\mu Q_\mu,\mathcal{M}_\mu^y\varepsilon)\\
&+\mu_s\tilde{\mu}(\partial_\mu^2 Q_\mu,{\mathcal{M}_\mu^y}\varepsilon)\nonumber\\
&=\tilde{\mu}\frac{\mu_s}{2\mu}(\partial_\mu Q_\mu,\nabla\cdot(y\mathcal{M}_\mu^y\varepsilon))+\mu_s\tilde{\mu}(\partial_\mu^2 Q_\mu,{\mathcal{M}_\mu^y}\varepsilon)=\tilde{\mu}\frac{\mu_s}{2\mu}(\partial_\mu Q_\mu,y\cdot\nabla\mathcal{M}_\mu^y\varepsilon)\\
&+\mu_s\tilde{\mu}(\partial_\mu^2 Q_\mu,{\mathcal{M}_\mu^y}\varepsilon).
\end{align*}
Hence, if we use Proposition \ref{proposition:comp2}, $|\mu_s|\leq2\mu^2$, and $\tilde{\mu}\leq\mu$ it holds
\begin{align*}
(F,{\mathcal{M}_\mu^y}\varepsilon)&\lesssim \mu^3\Big(\int_{\RR^2}Q_\mu| {\mathcal{M}_\mu^y}\varepsilon|^2dy\Big)^{\frac{1}{2}}\Big(\int_{\RR^2}\frac{|\partial_\mu^2 Q_\mu|^2}{Q_\mu}dy\Big)^{\frac{1}{2}}\\
&+\mu^2\Big(\int_{\RR^2}\frac{|\partial_\mu Q_\mu|^2|y|^2}{Q_\mu}dy\Big)^{\frac{1}{2}}\Big(\int_{\RR^2}Q_\mu|\nabla{\mathcal{M}_\mu^y}\varepsilon|^2dy\Big)^{\frac{1}{2}}\nonumber\\
&\lesssim (\sqrt{A}+\frac{1}{\delta})\mu^2+\delta\int_{\RR^2}Q_\mu|\nabla{\mathcal{M}_\mu^y}\varepsilon|^2dy
\end{align*}
Finally, combining all the previous estimates, we obtain
\begin{align}
\frac{1}{2}\frac{d}{ds}({\mathcal{M}_\mu^y}\varepsilon,\varepsilon)+\mu(K_2-1-\delta)({\mathcal{M}_\mu^y}\varepsilon,\varepsilon)\leq C(\delta,A)\mu^2,
\end{align}
with $\delta<<1$ and $C(A,\delta)\lesssim \frac{1}{\delta}(\sqrt{A}(1+\delta)+1)\leq A$ for $A$ sufficiently large, which concludes the energy estimates.
\end{proof}
Now we prove Proposition \ref{proposition:bootstrap},
\begin{proof}[Proof of Proposition \ref{proposition:bootstrap}]
As set before $\mathcal{E}(s)=({\mathcal{M}_\mu^y}\varepsilon,\varepsilon)$,
and now set $$K=K_2-1-\delta>1.$$
If we use that $\Big|\mu(s)-\frac{1}{2s}\Big|\leq\frac{C'}{s|\log s|}$
then,
\begin{align}
\frac{1}{2}\mathcal{E}'(s)+\frac{K}{2s}\mathcal{E}(s)\leq \frac{C(\delta,A)}{4s^2}.
\end{align}  
Hence, 
$$\frac{d}{ds}(s^{K}\mathcal{E}(s))\leq \frac{C(\delta,A)s^{K-2}}{2},$$
if we integrate it follows
$$s^{K}\mathcal{E}(s)\lesssim C(\delta,A)s^{K-1},$$
which concludes the proof, and we see easily that if we pick orriginally $A$ sufficiently large the constant in the bootstrap is getting smaller after each step.
\end{proof}
\section{ Proof of Corollary \ref{corollary:th_on_mass}}
\begin{proof}[Proof of Corollary \ref{corollary:th_on_mass}]
First, let us denote
\begin{align}\label{pmass}
m_u(r,s) =2\pi\int_0^r\tau u(\tau,t)d\tau,
\end{align}
where $r=|x|$. 
Hence, we obtain the following local equation on $m_u$:
\begin{align}\label{mass}
 \pd_t m_u &= \pd_{rr}m_u-\frac{\pd_r m_u}{r}+\frac{m_u\pd_r m_u}{r},\\
 m_u(t=0)&=m_0=\int_0^r \tau u_0(\tau)d\tau\nonumber\\
 m_u(r,t)&\longrightarrow 8\pi \mbox{ as }\quad r\rightarrow+\infty.
\end{align} 
The proof of the Corollary is just based on the fact that for radial data, we can use the equation on the partial mass and the comparison principle for radial solutions (See \cite{N}):
If for all $r\geq0$ \quad $m^2(t=0,r)\leq m^1(t=0,r)$ then for all $t>0$ and $r\geq0$ 
$$m^2(t,r)\leq m^1(t,r).$$

Let us choose a radially symmetric initial data $u_0^1$ satisfying the condition of Theorem \ref{theorem:main}, compactly supported and such that $u_0^1(0)>0$. Consequently, for $a>0$ large enough we can achieve that for all $r\geq0$,
$$ m_0(r)\leq m^1_0(r),$$
where
$$m_0^1(r)=2\pi\int_0^r\tau (u^1_0)_{\frac{1}{a}}(\tau)d\tau.$$
In addition, if we choose 
$$u_0^2=\frac{K}{\mu_0}Q\Big(\frac{x}{\sqrt{\mu_0}}\Big)e^{-\frac{x^2}{2}},$$
where $K>1$ is a constant chosen to insure that the mass of $u_0^2$ is $8\pi$.
Moroever, $u_0^2$ verifies the condition of Theorem \ref{theorem:main},
then, for $a$ large enough one can deduce easily that
$$m_0^2(r)\leq m_0(r),$$
where
$$m_0^2(r)=2\pi\int_0^r\tau (u^2_0)_{a}(\tau)d\tau.$$
Hence by applying the comparison principle it follows that
$$m_{u^2}(t)\leq m_u(t)\leq m_{u^1}(t),$$
where $u^2$ is the solution with initial data $(u^2_0)_{a}$ and
$u^1$ is the solution with initial data $(u^1_0)_{\frac{1}{a}}$.
In other words, because we can use the comparison principle we can bound the partial mass of any radial solution between 2 rescaled solutions for which Theorem \ref{theorem:main} apply.
To prove \eqref{part_mass_rate}, we use that $u^1$ and $u^2$ follow the dynamic of the solutions of Theorem \ref{theorem:main}.
Indeed, we can decompose $u^1$ and $u^2$ as in \eqref{decomp_u}, but to be more consistent we use the decomposition \eqref{decomp_v}:
$$u^i(x,t)=\frac{1}{\lambda_i^2}\tilde{Q}_{\mu_i}\Big(\frac{x}{\lambda_i},t\Big)+\frac{1}{\lambda_i^2}\varepsilon_i\Big(\frac{x}{\lambda_i},t\Big)\quad \mbox{ for } i\in\{1,2\},$$
where $\tilde{Q}_{\mu_i}$ is the approximate profile constructed in section 4, and $\mu_i$ is a function of $t$ fixed in section 5. 
Actually, with this decomposition we have 
$$\int_{\RR^2}\varepsilon_i dy=0.$$
Consequently,
\begin{align}
\int_0^{\frac{r}{\lambda_i}}\varepsilon_i(r')r' dr'=-\int_{\frac{r}{\lambda_i}}^\infty\varepsilon_i(r')r' dr'.
\end{align}
Moreover, we use that 
$$\tilde{Q}_{\mu_i}=Q_{\mu_i}-\tilde{\mu}_i\partial_{\mu_i} Q_{\mu_i},$$
where $\tilde{\mu}_i=\mu_i+O\Big(\frac{\mu_i}{|\log\mu_i|}\Big)$.

From Propositions \ref{proposition:approxQmu} and \ref{proposition:comp2}, we deduce that
\begin{align*}
8\pi-m_{\tilde{Q}_ {\mu_i}}&=\int_{\frac{r}{\lambda_i}}^\infty Q_ {\mu_i}(r')r'dr'-\tilde{{\mu_i}}\int_{\frac{r}{\lambda_i}}^\infty\partial_\mu Q_ {\mu_i} r'dr'\nonumber\\
&=\int_{\frac{r}{\lambda_i}}^\infty Q e^{{\mu_i}\phi_{T_1}+\sigma-{\mu_i}\frac{r'^2}{2}}r'dr'\\
&-\tilde{{\mu_i}}\int_{\frac{r}{\lambda_i}}^\infty \Big(\phi_{T_1}+\partial_ {\mu_i}\sigma-\frac{r'^2}{2}\Big)Q e^{{\mu_i}\phi_{T_1}+\sigma-{\mu_i}\frac{r'^2}{2}} r'dr',
\end{align*}
and there exists $0<c<1$, uniform with respect to ${\mu_i} $, such that
$$| {\mu_i}\phi_{T_1}+\sigma|\leq c {\mu_i}\frac{r^2}{2}.$$

Hence, for all $r>0$, we deduce that
\begin{align}
|8\pi-m_{\tilde{Q}_ {\mu_i}}|&\leq \int_{{\frac{r}{\lambda_i}}}^\infty Q e^{-{\mu_i}(1-c)\frac{r'^2}{2}}r'dr'+ {\mu_i}(1+c)\int_{\frac{r}{\lambda_i}}^\infty Q e^{-{\mu_i}(1-c)\frac{r'^2}{2}}r'dr'\nonumber\\
&\leq \frac{1+c}{1-c}\int_{\frac{r}{\lambda_i}}^\infty Q e^{-{\mu_i}(1-c)\frac{r'^2}{2}}r'dr'+ {\mu_i}\frac{(1+c)(1-c)}{1-c}\int_{\frac{r}{\lambda_i}}^\infty Q e^{-{\mu_i}(1-c)\frac{r'^2}{2}}r'dr'\nonumber\\
&\leq \frac{8(1+c)}{1-c}\Big[ \int_{\frac{r}{\lambda_i}}^\infty \frac{1}{r'^3} e^{-{\mu_i}(1-c)\frac{r'^2}{2}}dr'+\frac{{\mu_i}(1-c)}{2}\int_{\frac{r}{\lambda_i}}^\infty \frac{1}{r'} e^{-{\mu_i}(1-c)\frac{r'^2}{2}}dr'\Big].
\end{align}

By integration by parts we obtain that,

$$\frac{{\mu_i}(1-c)}{2}\int_{\frac{r}{\lambda_i}}^\infty \frac{1}{r'} e^{-{\mu_i}(1-c)\frac{r'^2}{2}}dr'=\frac{\lambda_i^2e^{-{\mu_i}(1-c)\frac{r^2}{2}}}{2r^2}-\int_{\frac{r}{\lambda_i}}^\infty \frac{1}{r'^3} e^{-{\mu_i}(1-c)\frac{r'^2}{2}}dr'.$$

Consequently, we get for all $r>0$ that

\begin{align}
|8\pi-m_{\tilde{Q}_{\mu_i}}|&\lesssim \frac{\lambda_i^2e^{-{\mu_i}(1-c)\frac{r^2}{2\lambda_i}}}{r^2}.
\end{align}

From Lemma \ref{lemma:mu_law} and Theorem \ref{theorem:main} we get that:
$$\mu_i(t)=\frac{\sqrt{I_i}}{t\log(2t+1)+O(t\log\log(2t+1))},$$
and
$$\lambda_i(t)=R(t)\mu_i(t), \mbox{ with } R(t)=\sqrt{2t+1}.$$

Hence,
\begin{align}
|8\pi-m_{\tilde{Q}_{\mu_i}}|&\lesssim \frac{\lambda_i^2 e^{-C_1\frac{r^2}{t}}}{r^2}.
\end{align}

Since, $\int_{\RR^2}\varepsilon dy=0$, we deduce that
\begin{align}
\int_{\frac{r}{\lambda_i}}^\infty\varepsilon_i(r')r' dr'&\lesssim\Big(\int_{\frac{r}{\lambda_i}}^\infty Q_{\mu_i}(r')dr'\Big)^{\frac{1}{2}}\|\varepsilon_i\|_{L^2_{Q_{\mu_i}}}\lesssim\Big(\int_{\frac{r}{\lambda_i}}^\infty Q(r')dr'\Big)^{\frac{1}{2}}\|\varepsilon_i\|_{L^2_{Q_{\mu_i}}}\nonumber\\
&\lesssim\frac{\lambda_i^2}{\lambda_i^2+r^2}\|\varepsilon_i\|_{L^2_{Q_\mu}}\lesssim \frac{\lambda_i^2}{(\lambda_i^2+r^2)t|\log t|}.
\end{align}
Consequently,
\begin{align*}
8\pi-m_{u^2}(r,t)&=8\pi-m_{\tilde{Q}_{\mu_2}}-m_{\varepsilon_2}\lesssim\frac{\lambda_2^2e^{-C_2\frac{r^2}{2}}}{r^2}+m_{\varepsilon_2}\Big(\frac{r}{\lambda_2},t\Big)\\
&\lesssim\frac{\lambda_2^2e^{-C_2\frac{r^2}{2}}}{r^2}+\frac{\lambda_2^2}{(\lambda_2^2+r^2)t|\log t|},
\end{align*}
and
$$8\pi-m_{u^1}(r,t)\gtrsim -\frac{\lambda_1^2e^{-C_1\frac{r^2}{2}}}{r^2}-\frac{\lambda_1^2}{(\lambda_1^2+r^2)t|\log t|},$$

with $\lambda_i(t)=\frac{\sqrt{I_i}}{\sqrt{\log(2t+1)+O(\log(\log(2t+1))}}$,
where $I_i$ depends on $a$ and the second moment of $u$.
This concludes the proof.
\end{proof}

\section{Appendix}

\begin{proposition}\label{Proposition:Hardyineq}
Let $f\in H^1(Q_\mu dy)$ be such that $(f,\phi_{\Lambda Q_\mu}Q_\mu)_{L^2}=0$,
then there exists $C$ uniform with respect to $\mu$ such that
\begin{align}\label{Hardyineq}
\int_{\RR^2}|f|^2\frac{|y|^2}{(1+|y|^2)^2}Q_\mu dy\leq C\int_{\RR^2}Q_\mu|\nabla f|^2dy.
\end{align}
\end{proposition}
\begin{proof}
We first prove the following inequality,
\begin{align}\label{subpoincineq}
\int_{\RR^2}|\nabla f|^2Q_\mu dy+\int_{B}|f|^2Q_\mu dy\geq \frac{3}{4}\int_{\RR^2}|f|^2\frac{|y|^2}{(1+|y|^2)^2}Q_\mu dy.
\end{align}
Indeed, let $\gamma$ be a positive constant that we will fix later,
\begin{align*}
\int_{\RR^2}\Big|\nabla f-\gamma \frac{y f}{1+|y|^2}\Big|^2Q_\mu dy&=\int_{\RR^2}|\nabla f|^2Q_\mu dy-2\gamma\int_{\RR^2}\frac{f y\cdot\nabla f}{1+|y|^2} Q_\mu dy\\
&+\gamma^2\int_{\RR^2}\frac{|f|^2|y|^2}{(1+|y|^2)^2}Q_\mu dy\geq0.
\end{align*}
By integration by parts, we deduce
\begin{align*}
-2\gamma\int_{\RR^2}\frac{f y\cdot\nabla f}{1+|y|^2} Q_\mu dy&=\gamma\int_{\RR^2}\frac{|f|^2\nabla Q_\mu\cdot y}{1+|y|^2}dy+2\gamma\int_{\RR^2}\frac{|f|^2 }{1+|y|^2} Q_\mu dy\\
&-2\gamma\int_{\RR^2}\frac{|f|^2 |y|^2}{(1+|y|^2)^2} Q_\mu dy.
\end{align*}
Since, 
\begin{align*}
\nabla Q_\mu\cdot y&=(y\cdot\nabla\phi_{Q}+\mu y\cdot\nabla\phi_{T_1}+y\cdot\nabla\sigma-\mu|y|^2)Q_\mu\\
&=(-\frac{4|y|^2}{1+|y|^2}+\mu y\cdot\nabla\phi_{T_1}+y\cdot\nabla\sigma-\mu|y|^2)Q_\mu,
\end{align*}
$$\phi_{T_1}'(r)=O(r^3\log r)\indic_{\{ r\leq 1\}}+O\Big(\frac{\log r}{r}\Big)\indic_{\{ r\geq 1\}},$$
and
$$\partial_r\sigma\lesssim\min\Big(\mu^2r|\log<r>|,\mu\frac{|\log<r>|}{r}\Big),$$
it follows,
\begin{align}
&\gamma\int_{\RR^2}\frac{|f|^2\nabla Q_\mu\cdot y}{1+|y|^2}dy+2\gamma\int_{\RR^2}\frac{|f|^2 }{1+|y|^2} Q_\mu dy\nonumber\\
&=\gamma\int_{\RR^2}\frac{|f|^2}{1+|y|^2}\Big[\frac{2(1-|y|^2)}{1+|y|^2}-\mu\underbrace{\Big(|y|^2-y\cdot\nabla\phi_{T_1}-\frac{y\cdot\nabla\sigma}{\mu}\Big)}_{>0}\Big]Q_\mu dy\leq 2\gamma\int_{B}Q_\mu |f|^2 dy.
\end{align}
Hence, if we combine all the previous inequalities we obtain
\begin{align}
\int_{\RR^2}|\nabla f|^2Q_\mu dy+2\gamma\int_{B}Q_\mu |f|^2 dy\geq(2\gamma-\gamma^2)\int_{\RR^2}\frac{|f|^2|y|^2}{(1+|y|^2)^2}Q_\mu dy,
\end{align}
if we select $\gamma=\frac{1}{2}$ then \eqref{subpoincineq} follows.
Now to show \eqref{Hardyineq} we prove that if $(f,\phi_{\Lambda Q_\mu}Q_\mu)_{L^2}=0$,
then
\begin{align}\label{Poincineq}
\int_{B}Q_\mu|f|^2\leq C\int_{\RR^2}Q_\mu|\nabla f|^2 dy,
\end{align}
with $C$ uniform with respect to $\mu$.
We first by contradiction prove \eqref{Poincineq} for fixed $\mu$.
Assume that \eqref{Poincineq} is false, then there exists $f_n\in H^1(\RR^2,Q_{\mu} dy)$ such that
$$\int_{\RR^2}Q_{\mu}|\nabla f_n|^2 dy\leq \frac{1}{n},$$
$$\int_{B}Q_\mu|f_n|^2 dy=1,$$
and 
$$(f_n,\phi_{\Lambda Q_{\mu}}Q_{\mu})_{L^2}=0.$$
Hence, up to a subsequence $f_n \longrightarrow f_\infty$ in $L^2_{loc}$ where $f_\infty$ is a constant.
In addition, by \eqref{subpoincineq} we deduce
$$\int_{\RR^2}\frac{|f_n|^2|y|^2}{(1+|y|^2)^2}Q_{\mu_n} dy\leq 1.$$
And  since $\phi_{\Lambda_{Q_{\mu}}}(y)=O\bigl(\frac{1}{1+|y|^2}\bigr)$
we can pass to the limit in
\begin{align}
(f_n,\phi_{\Lambda Q_{\mu}}Q_{\mu})_{L^2}=0,
\end{align}
and deduce that
$$(f_\infty,\phi_{\Lambda Q_{\mu}}Q_{\mu})_{L^2}=0.$$
Since $(1,\phi_{\Lambda Q_{\mu}}Q_{\mu})_{L^2}\neq0$, we get $f_\infty=0$.
However,
$$\int_{B}|f_\infty|^2Q_\mu dy=1,$$
which contradicts $f_\infty=0$.
To prove the uniformity of $C$ in \eqref{Poincineq}, we argue by contradiction. Suppose that $C$ is not uniform with respect to $\mu$, then there exist $C_n>0$, $\mu_n>0$, and $f_n\in H^1(\RR^2,Q_{\mu_n} dy)$ such that
$$C_n\rightarrow+\infty,$$
$$\mu_n\rightarrow\mu^*,$$
with $\mu^*\geq0$.
$$\int_{\RR^2}Q_{\mu_n}|\nabla f_n|^2 dy=\frac{1}{C_n},$$
$$\int_{B}Q_{\mu_n}|f_n|^2 dy=1,$$
and 
$$(f_n,\phi_{\Lambda Q_{\mu_n}}Q_{\mu_n})_{L^2}=0.$$
Actually, $\mu^*=0$ is the most difficult case, since when $\mu\rightarrow 0$, $Q_\mu$ looses its exponential decay at infinity.
However, thanks to the fact that $\phi_{\Lambda Q}=O\bigl(\frac{1}{1+|y|^2}\bigr)$
has extrat decay, one can still pass to the limit in $(f_n,\phi_{\Lambda Q_{\mu_n}}Q_{\mu_n})_{L^2}=0$.
Indeed, for $R>0$, we have
\begin{align}
0=(f_n,\phi_{\Lambda Q_{\mu_n}}Q_{\mu_n})_{L^2}&=\int_{B_R} f_n\phi_{\Lambda Q_{\mu_n}}Q_{\mu_n}dy +\int_{\RR^2\setminus B_R}f_n\phi_{\Lambda Q_{\mu_n}}Q_{\mu_n}dy,
\end{align}
and
$$\int_{B_R} f_n\phi_{\Lambda Q_{\mu_n}}Q_{\mu_n}dy\longrightarrow \int_{B_R} f_\infty\phi_{\Lambda Q}Qdy.$$
Moreover, thanks to \eqref{subpoincineq} we get that
\begin{align}
\int_{\RR^2\setminus B_R}f_n\phi_{\Lambda Q_{\mu_n}}Q_{\mu_n}dy&\lesssim \Big(\int_{\RR^2\setminus B_R}\frac{|f_n|^2|y|^2}{(1+|y|^2)^2}Q_{\mu_n} dy\Big)^{\frac{1}{2}}\Big(\int_{\RR^2\setminus B_R}|\phi_{\Lambda Q_{\mu_n}}|^2|y|^2Q_{\mu_n} dy\Big)^{\frac{1}{2}}\nonumber\\
&\lesssim\frac{1}{R^2}\Big(\int_{\RR^2\setminus B_R}\frac{|f_n|^2|y|^2}{(1+|y|^2)^2}Q_{\mu_n} dy\Big)^{\frac{1}{2}}.
\end{align}
Since 
$$\Big(\int_{\RR^2\setminus B_R}\frac{|f_n|^2|y|^2}{(1+|y|^2)^2}Q_{\mu_n} dy\Big)^{\frac{1}{2}},$$
is uniformly bounded with respect to $n$, we deduce that,
$$\int_{\RR^2\setminus B_R}f_n\phi_{\Lambda Q_{\mu_n}}Q_{\mu_n}dy\longrightarrow 0,$$
when $R\rightarrow+\infty$ uniformly in $n$.
For the other terms, one can pass to the limit as it has been done before and reach the same contradiction which concludes the proof.
\end{proof}
\begin{lemma}\label{lemma:PCineq}
Let $f\in H^1(\RR^2,Q_\mu dy)$ then there exists $c'>0$ uniform with respect to $\mu$ such that:
$$\int_{\RR^2}Q_\mu |y|^2|f|^2dy\leq \frac{c'}{\mu^2}\int_{\RR^2}Q_\mu|\nabla f|^2dy+\frac{c'}{\mu}\int_{B_1}Q_\mu|f|^2dy,$$
where $B_1$ is the unit ball in $\RR^2$.
\end{lemma}
\begin{proof}
The proof is based on the following identity with $\gamma$ a positive constant that will be fixed at the end of the proof,
\begin{align}
\int_{\RR^2}|\nabla f-\gamma y f|^2Q_\mu dy=\int_{\RR^2}|\nabla f|^2Q_\mu dy-2\gamma\int_{\RR^2}f y\cdot\nabla f Q_\mu dy+\gamma^2\int_{\RR^2}|f|^2|y|^2Q_\mu dy\geq0.\nonumber
\end{align}
By integrating by parts we get
$$-2\gamma\int_{\RR^2}f y\cdot\nabla f Q_\mu dy=\gamma\int_{\RR^2}|f|^2\nabla\cdot( y Q_\mu) dy=2\gamma\int_{\RR^2}|f|^2 Q_\mu dy+\gamma\int_{\RR^2}|f|^2 y\cdot\nabla Q_\mu dy.$$
Hence,
\begin{align}
\int_{\RR^2}|\nabla f|^2Q_\mu dy\geq -\gamma^2\int_{\RR^2}|f|^2|y|^2Q_\mu dy-2\gamma\int_{\RR^2}|f|^2 Q_\mu dy-\gamma\int_{\RR^2}|f|^2 y\cdot\nabla Q_\mu dy.
\end{align}
 From Proposition \ref{proposition:comp2} there exists $0<c<1$ uniform with respect to $\mu$ such that 
$$y\cdot\nabla\phi_{T_1}+\frac{y\cdot\nabla\sigma}{\mu}\leq c|y|^2.$$ 

Since $y\cdot\nabla Q_\mu=y\cdot\nabla\phi_{Q_\mu}-\mu|y|^2=-\frac{4|y|^2}{1+|y|^2}-\mu(|y|^2-y\cdot\nabla\phi_{T_1}-\frac{y\cdot\nabla\sigma}{\mu})$

we get $$\int_{\RR^2}|\nabla f|^2Q_\mu dy\geq -\gamma(\gamma-(1-c)\mu)\int_{\RR^2}|f|^2|y|^2Q_\mu dy-2\gamma\int_{\RR^2}\frac{1-|y|^2}{1+|y|^2}|f|^2Q_\mu dy.$$

If we choose $\gamma=\frac{(1-c)\mu}{2}$ the conclusion follows.
\end{proof}
\begin{proposition}\label{proposition:difbetE_hatE}
If $(\varepsilon,1)_{L^2}=(\varepsilon,|\cdot|^2)_{L^2}=0$, $\varepsilon=\alpha_\mu\Lambda Q_\mu+\hat{\varepsilon}$ and  $\int_{\RR^2}Q_\mu\mathcal{M}_\mu^y\hat{\varepsilon}dy=0$ then
\begin{align}
(\mathcal{L}_\mu^y\varepsilon,\mathcal{M}_\mu^y\varepsilon)=(\mathcal{L}_\mu^y\hat{\varepsilon},\mathcal{M}_\mu^y\hat{\varepsilon})+O(\mu^3|\log\mu|),
\end{align}
\begin{align}
({\mathcal{M}_\mu^y}\varepsilon,\varepsilon)=({\mathcal{M}_\mu^y}\hat{\varepsilon},\hat{\varepsilon})+O(\mu^2|\log\mu|).
\end{align}
\end{proposition}
\begin{proof}
The proof is just a consequence of the following identities $\mathcal{M}_\mu^y(\Lambda Q_\mu)=2-\frac{M}{2\pi}-\mu|y|^2$ and $\mathcal{L}_\mu^y(\Lambda Q_\mu)=-2\mu\Lambda Q_\mu$ and the orhtogonality conditions on $\varepsilon$.
Indeed, if we replace  $\varepsilon$ in $({\mathcal{M}_\mu^y}\varepsilon,\varepsilon)$ by $\alpha_\mu\Lambda Q_\mu+\hat{\varepsilon}$, it follows
\begin{align}
({\mathcal{M}_\mu^y}\varepsilon,\varepsilon)&= ({\mathcal{M}_\mu^y}\hat{\varepsilon},\varepsilon)+\alpha_\mu({\mathcal{M}_\mu^y}\Lambda Q_\mu,\varepsilon)=({\mathcal{M}_\mu^y}\hat{\varepsilon},\varepsilon)=({\mathcal{M}_\mu^y}\hat{\varepsilon},\hat{\varepsilon})+\alpha_\mu({\mathcal{M}_\mu^y}\hat{\varepsilon},\Lambda Q_\mu)\nonumber\\
&=({\mathcal{M}_\mu^y}\hat{\varepsilon},\hat{\varepsilon})+\alpha_\mu(\hat{\varepsilon},2-\frac{M}{2\pi}-\mu|\cdot|^2).
\end{align}
Since $\int_{\RR^2}\varepsilon dy=\int_{\RR^2}\Lambda Q_\mu dy=0$ it implies
$$\int_{\RR^2}\hat{\varepsilon} dy=0.$$
Hence,
\begin{align}
({\mathcal{M}_\mu^y}\varepsilon,\varepsilon)&=({\mathcal{M}_\mu^y}\hat{\varepsilon},\hat{\varepsilon})-\alpha_\mu\mu(\hat{\varepsilon},|\cdot|^2).
\end{align}
Moreover, since $\int_{\RR^2}\varepsilon |y|^2 dy=0$ it follows
$$\int_{\RR^2}\hat{\varepsilon}|y|^2 dy=-\alpha_\mu(\Lambda Q_\mu,|\cdot|^2)=2\alpha_\mu(Q_\mu,|\cdot|^2)=\alpha_\mu\frac{M}{\pi}|\log\mu|+O(1).$$
Then using that $\alpha_\mu\lesssim \sqrt{\mu}$, we deduce
\begin{align}
({\mathcal{M}_\mu^y}\varepsilon,\varepsilon)&=({\mathcal{M}_\mu^y}\hat{\varepsilon},\hat{\varepsilon})+O(\mu^2|\log\mu|).
\end{align}
Now we prove the second identity and we replace $\varepsilon$ in $(\mathcal{L}_\mu^y\varepsilon,\mathcal{M}_\mu^y\varepsilon)$ by $\hat{\varepsilon}+\alpha_\mu\Lambda Q_\mu$:
\begin{align*}
(\mathcal{L}_\mu^y\varepsilon,\mathcal{M}_\mu^y\varepsilon)&=(\mathcal{L}_\mu^y\hat{\varepsilon},\mathcal{M}_\mu^y\varepsilon)+\alpha_\mu(\mathcal{L}_\mu^y\Lambda Q_\mu,\mathcal{M}_\mu^y\varepsilon)=(\mathcal{L}_\mu^y\hat{\varepsilon},\mathcal{M}_\mu^y\varepsilon)-2\mu\alpha_\mu(\Lambda Q_\mu,\mathcal{M}_\mu^y\varepsilon)\nonumber\\
&=(\mathcal{L}_\mu^y\hat{\varepsilon},\mathcal{M}_\mu^y\varepsilon)-2\mu\alpha_\mu(2-\frac{M}{2\pi}-\mu|\cdot|^2,\varepsilon)=(\mathcal{L}_\mu^y\hat{\varepsilon},\mathcal{M}_\mu^y\varepsilon)\\
&=(\mathcal{L}_\mu^y\hat{\varepsilon},\mathcal{M}_\mu^y\hat{\varepsilon})+\alpha_\mu(\mathcal{L}_\mu^y\hat{\varepsilon},\mathcal{M}_\mu^y\Lambda Q_\mu)\nonumber\\ &=(\mathcal{L}_\mu^y\hat{\varepsilon},\mathcal{M}_\mu^y\hat{\varepsilon})+\alpha_\mu(\mathcal{M}_\mu^y\hat{\varepsilon},\mathcal{L}_\mu^y\Lambda Q_\mu)=(\mathcal{L}_\mu^y\hat{\varepsilon},\mathcal{M}_\mu^y\hat{\varepsilon})-2\mu\alpha_\mu(\mathcal{M}_\mu^y\hat{\varepsilon},\Lambda Q_\mu)\nonumber\\
&=(\mathcal{L}_\mu^y\hat{\varepsilon},\mathcal{M}_\mu^y\hat{\varepsilon})+\alpha_\mu(\mathcal{M}_\mu^y\hat{\varepsilon},\mathcal{L}_\mu^y\Lambda Q_\mu)=(\mathcal{L}_\mu^y\hat{\varepsilon},\mathcal{M}_\mu^y\hat{\varepsilon})-2\mu\alpha_\mu(\hat{\varepsilon},2-\frac{M}{2\pi}-\mu|\cdot|^2).
\end{align*}
As before we use that $\int_{\RR^2}\hat{\varepsilon} dy=0$, $\int_{\RR^2}\hat{\varepsilon}|y|^2 dy=\alpha_\mu\frac{M}{\pi}|\log\mu|+O(1)$ and $\alpha_\mu\lesssim \sqrt{\mu}$ which imply 
\begin{align}
(\mathcal{L}_\mu^y\varepsilon,\mathcal{M}_\mu^y\varepsilon)=(\mathcal{L}_\mu^y\hat{\varepsilon},\mathcal{M}_\mu^y\hat{\varepsilon})+O(\mu^3|\log\mu|).
\end{align}
This concludes the proof.
\end{proof}
The following estimations have been proved already in \cite{RS} for $\varepsilon\in L^2_Q(\RR^2)$ but since $Q_\mu<Q$ we get that $L^2_{Q_\mu}(\RR^2)\hookrightarrow L^2_Q(\RR^2)$ and we deduce easily from their proposition the following one:
\begin{proposition}\label{proposition:ineqpot}
Let $\varepsilon\in L^2_{Q_\mu}(\RR^2)$ then 
$$\|\nabla\phi_{\varepsilon}\|_{L^4}\leq C\|\varepsilon\|_{L^2_{Q_\mu}}.$$
If in addition $\int_{\RR^2}\varepsilon dy=0$, then 
$$\|\phi_{\varepsilon}\|_{L^\infty}\leq C \|\varepsilon\|_{L^2_{Q_\mu}},$$
\begin{equation}\label{gradphi1}
\|\nabla \phi_\varepsilon \|_{L^2}\leq C\| \varepsilon \|_{L^2_{Q_\mu}}.
\end{equation}
\end{proposition}
\begin{lemma}\label{lemma:poissonbound}
Let $\varepsilon\in L^2_{Q_\mu}(\RR^2)$ such that $\int_{\RR^2}\varepsilon (y)dy=\int_{\RR^2}y_i\varepsilon (y)dy=\int_{\RR^2}|y|^2\varepsilon (y)dy=0$, then
$$\|\phi_{\varepsilon}\|_{L^2}\leq C \|\varepsilon\|_{L^2_{Q_\mu}},$$
with $C$ uniform with respect to $\mu$.
\end{lemma}
\begin{proof}
For $|x|\leq 1$ we have
\begin{align}
|\phi_{\varepsilon}(x)|&\leq \int_{\RR^2}|\log(|x-y|)||\varepsilon(y)| dy\nonumber\\
&\lesssim \|\varepsilon\|_{L^2_{Q_\mu}},
\end{align}
which implies
$$\int_{|x|\leq 1}|\phi_{\varepsilon}(x)|^2dx\lesssim \|\varepsilon\|_{L^2_{Q_\mu}}^2.$$
For $|x|\geq 1$ we have 
\begin{align}
\Big|\phi_{\varepsilon}(x)-\frac{\log(|x|)}{2\pi}\int_{\RR^2}\varepsilon(y)dy\Big|=\Big|\int_{\RR^2}\log\Big(\frac{|x-y|}{|x|}\Big)\varepsilon(y)dy\Big|. 
\end{align}
We separate the integral in teo pieces.
Indeed, we look first at $\{|y|\leq|x|\}$,
\begin{align}
\Big|\int_{\{|y|\leq|x|\}}\log\Big(\frac{|x-y|}{|x|}\Big)\varepsilon(y)dy\Big|&=\frac{1}{2}\Big|\int_{\{|y|\leq|x|\}}\log\Big(\frac{|x-y|^2}{|x|^2}\Big)\varepsilon(y)dy\Big|\nonumber\\
& =\Big|\int_{\{|y|\leq|x|\}}\log\Big(1+\frac{|y|^2}{|x|^2}-\frac{2x\cdot y}{|x|^2}\Big)\varepsilon(y)dy\Big|\nonumber\\
&=\Big|\int_{\{|y|\leq|x|\}}\Big(\frac{|y|^2}{|x|^2}-\frac{2x\cdot y}{|x|^2}+O\Big(\frac{|y|^{2-\eta}}{|x|^{2-\eta}}\Big)\Big)\varepsilon(y)dy\Big|,
\end{align}
with $\eta$ small.
Hence, by using $\int_{\RR^2}y_i\varepsilon (y)dy=\int_{\RR^2}|y|^2\varepsilon (y)dy=0$ we deduce 
\begin{align}
\Big|\int_{\{|y|\leq|x|\}}\log\Big(\frac{|x-y|}{|x|}\Big)\varepsilon(y)dy\Big|\lesssim \int_{\{|y|\leq|x|\}}\frac{|y|^{2-\eta}}{|x|^{2-\eta}}|\varepsilon(y)|dy\lesssim \frac{\|\varepsilon\|_{L^2_{Q_\mu}}}{|x|^{2-\eta}}.
\end{align}
Now we look the other part of the integral,
\begin{align}
\Big|\int_{\{|y|\geq|x|\}}\log\Big(\frac{|x-y|}{|x|}\Big)\varepsilon(y)dy\Big|
\end{align}
Since, $|x-y|^2\leq2(|x|^2+|y|^2)$ and $\log(1+t)\leq t^{1-\eta}$ with $\eta>0$ sufficiently small, it follows
\begin{align}
\Big|\int_{\{|y|\geq|x|\}}\log\Big(\frac{|x-y|}{|x|}\Big)\varepsilon(y)dy\Big|&=\frac{1}{2}\Big|\int_{\{|y|\geq|x|\}}\log\Big(\frac{|x-y|^2}{|x|^2}\Big)\varepsilon(y)dy\Big|\nonumber\\
& \leq \frac{1}{2}\Big|\int_{\{|y|\geq|x|\}}\log\Big(2+2\frac{|y|^2}{|x|^2}\Big)\varepsilon(y)dy\Big|\nonumber\\
&\leq\frac{1}{2}\Big|\int_{\{|y|\geq|x|\}}\frac{|y|^{2-2\eta}}{|x|^{2-2\eta}}\varepsilon(y)dy\Big|.
\end{align}
Hence, we deduce
\begin{align}
\int_{|x|\geq 1}|\phi_{\varepsilon}(x)|^2dx\lesssim \|\varepsilon\|^2_{L^2_{Q_\mu}},
\end{align}
which concludes the proof.
\end{proof}


\def\cprime{$'$}

\end{document}